\documentclass[11pt,reqno,a4paper]{amsart}

\usepackage{amsmath}
\usepackage{amssymb}
\usepackage{graphicx}
\usepackage{mathtools}
\usepackage{bm,upgreek}
\usepackage{enumerate}

\usepackage[color]{changebar}
\newtheorem{thm}{Theorem}
\newtheorem*{thm*}{Theorem}
\newtheorem{prop}{Proposition}[section]
\newtheorem{lem}[prop]{Lemma}
\newtheorem*{lem*}{Lemma}

\newtheorem{defn}[prop]{Definition}
\theoremstyle{remark}
\newtheorem{rema}[prop]{Remark}

\def\mbb{\mathbb}
\def\mb{\mathbf}

\def\mr{\mathrm}
\def\mc{\mathcal}

\newlength{\equwidth}
\newcommand{\crd}{\mbox{$                                     
\begin{picture}(9,8)(1.6,0.15)
\put(1,0.2){\mbox{$ D \hspace{-7.8pt} /$}}
\end{picture}$}}
\def\dirac{\crd}

\DeclareMathOperator{\Ric}{Ric}

\def\al{\alpha}

\def\ga{\gamma}
\def\de{\delta}

\def\la{\lambda}

\def\Rho{\mbox{\textsf{P}}}
\def\si{\sigma}

\def\Ups{\Upsilon}

\def\om{\omega}
\def\Ga{\Gamma}

\def\Th{\Theta}
\def\La{\Lambda}

\def\embed{\hookrightarrow}

\def\cinf{\ensuremath{\mathrm{C}^\infty}}

\def\im{\mathrm{im}\ }
\def\dim{\mathrm{dim\ }}
\def\G{\mathcal{G}}

\DeclareMathOperator{\Spin}{Spin}

\def\wt{\widetilde}

\def\rr{\ensuremath{\mathbb{R}}}

\def\SL{\ensuremath{\mathrm{SL}}}

\def\SO{\ensuremath{\mathrm{SO}}}

\def\sl{\ensuremath{\mathfrak{sl}}}
\def\so{\ensuremath{\mathfrak{so}}}

\def\ce{\ensuremath{\mathcal{E}}}

\def\sideremark#1{\ifvmode\leavevmode\fi\vadjust{\vbox to0pt{\vss
 \hbox to 0pt{\hskip\hsize\hskip1em
 \vbox{\hsize3cm\tiny\raggedright\pretolerance10000
 \noindent #1\hfill}\hss}\vbox to8pt{\vfil}\vss}}}%

\def\today{\ifcase\month\or
 January\or February\or March\or April\or May\or June\or
 July\or August\or September\or October\or November\or December\fi
 \space\number\day, \number\year}

\usepackage{amssymb}
\usepackage{paralist} 
\usepackage{graphicx} 
\usepackage{hyperref}
\usepackage[all]{xy}

\usepackage{hyphenat}

\def\mb{\mathbf} 

\def\mr{\mathrm}
\def\mc{\mathcal}

\def\ee{\ensuremath{\mathrm{e}}}

\def\d{\operatorname{d}\!}

\def\Vertical{V}

\usepackage{tensor}
\newcommand{\ind}{\indices}
\newcommand{\parderv}[2] {\frac{\partial#1}{\partial#2}}

\newcommand*\patchAmsMathEnvironmentForLineno[1]{%
  \expandafter\let\csname old#1\expandafter\endcsname\csname #1\endcsname
  \expandafter\let\csname oldend#1\expandafter\endcsname\csname end#1\endcsname
  \renewenvironment{#1}%
     {\linenomath\csname old#1\endcsname}%
     {\csname oldend#1\endcsname\endlinenomath}}%
\newcommand*\patchBothAmsMathEnvironmentsForLineno[1]{%
  \patchAmsMathEnvironmentForLineno{#1}%
  \patchAmsMathEnvironmentForLineno{#1*}}%
\AtBeginDocument{%
\patchBothAmsMathEnvironmentsForLineno{equation}%
\patchBothAmsMathEnvironmentsForLineno{align}%
\patchBothAmsMathEnvironmentsForLineno{flalign}%
\patchBothAmsMathEnvironmentsForLineno{alignat}%
\patchBothAmsMathEnvironmentsForLineno{gather}%
\patchBothAmsMathEnvironmentsForLineno{multline}%
}

\usepackage{lineno}
 
\title{Conformal Patterson--Walker metrics}

\author[Hammerl, Sagerschnig, \v{S}ilhan, Taghavi-Chabert, \v{Z}\'adn\'ik]{Matthias Hammerl, Katja Sagerschnig,
 Josef \v{S}ilhan,
 Arman Taghavi-Chabert and Vojt\v{e}ch \v{Z}\'adn\'ik}
\address{\flushleft M. H.: University of Vienna, Faculty of Mathematics, Oskar-Morgenstern-Platz 1, 1090 Wien, Austria
\newline K. S.: Center for Theoretical Physics PAS, Al. Lotnik\'ow 32/46 02-668 Warszawa, Poland
\newline J. \v S.:
Masaryk University, Faculty of Science, Kotl\'{a}\v{r}sk\'{a} 2, 61137 Brno, Czech Republic
\newline 
A. T.-C.:  Department of Mathematics, Faculty of Arts and Sciences, American University of Beirut, P.O. Box 11-0236, Riad El Solh, Beirut 1107 2020, Lebanon
\newline 
V. \v{Z}: Masaryk University, Faculty of Education, Po\v{r}\'\i\v{c}\'\i\ 31, 60300 Brno, Czech Republic
}
\email{matthias.hammerl@univie.ac.at, katja@cft.edu.pl, silhan@math.muni.cz, at68@aub.edu.lb, zadnik@mail.muni.cz}

\date{\today}



\begin{document}

\begin{abstract}
The classical Patterson--Walker construction of a split\hyp{}signature (pseudo-)Riemannian structure from a given torsion-free affine connection is generalized to a construction of a split\hyp{}signature conformal structure from a given projective class of connections. A characterization of the induced structures is obtained. We achieve a complete description of Einstein metrics in the conformal class formed by the Patterson--Walker metric. Finally, we describe all symmetries of the conformal Patterson--Walker metric. In both cases we obtain descriptions in terms of geometric data on the original structure.
\end{abstract}

\maketitle



\section{Introduction}

Given a torsion-free affine connection $D$ on a smooth $n$-dimensional 
manifold
$M$, the classical Patterson--Walker construction \cite{patterson-walker} 
yields
a split-signature $(n,n)$ pseudo-Riemannian metric $g$ on the 
total space
of the cotangent bundle $T^*M$. 
The metric $g$ is determined by the natural pairing of 
the vertical distribution $V$ of $T^*M$ and the horizontal distribution $H\cong TM$ on $T^*M$. 
In particular, $V$ and $H$ (as determined by $D$) are totally isotropic with respect to $g$. Such metrics are endowed with a parallel pure spinor and a homothety, and satisfy an integrability condition on the Riemann curvature tensor. 
We shall show  in section \ref{sec-Riem-ext} that Patterson--Walker metrics are locally characterized by these data.

When $n=2$, this construction is generalised in \cite{dunajski-tod} where a \emph{conformal class} of Patterson--Walker metrics is assigned to a \emph{projective class} of volume-preserving torsion-free affine connections. As we shall see, this extends to any dimension. In order to accommodate projective invariance in this construction, we must replace $T^*M$ by the density-valued cotangent bundle $T^*M(2)$.
Recall that the projective class $\mb{p}$ containing $D$ is formed by all torsion-free affine connections which share the same geodesics (as unparametrized
curves) as $D$. We shall suppose in addition that $D$ preserves a volume form on $M$, and as such will be referred to as special. Then special connections $D, \widehat{D} \in \mb{p}$ give rise to 
Patterson--Walker metrics $g, \hat{g}$ on $T^*M(2)$
which are conformally related, i.e.\ $\hat{g}=\ee^{2f} g$ for some smooth function $f$ on $M$. In other words, 
the projective
structure $(M,\mb{p})$ induces a split-signature conformal structure
$(T^*M(2),\mb{c})$, see section~\ref{sec-construction} for details.

Notice that certain geometrical data are to be expected on the conformal manifold
$(\wt{M},\mb{c})$ induced from a projective class $(M,\mb{p})$.
Firstly, there is a distinguished vector field $k$ corresponding to the 
Euler vector field
on $T^*M(2)$. Secondly, there is an $n$-dimensional integrable 
distribution $\Vertical$ on $\wt{M}$ corresponding to the vertical subbundle of $T^*M(2)$.
In fact, this distribution can be conveniently defined via a distinguished pure 
spinor field $\chi$ annihilating $\Vertical$. Here purity of $\chi$ 
corresponds to $\Vertical=\ker\chi$ being maximally isotropic. Further, one expects an integrability 
condition imposed
on the curvature of metrics in $\mb{c}$ and this we shall formulate
in terms of the (conformally invariant) Weyl tensor $\wt{W}_{abcd}$ of 
$\mb{c}$. 
Our characterization result, proved in section~\ref{sec-char}, is then
\begin{thm}\label{thm-char}
  A conformal spin structure $\mb c$ of split signature $(n,n)$ on a manifold $\wt{
  M}$ is locally induced by an $n$-dimensional projective structure as a conformal Patterson--Walker metric if and only if the following properties are satisfied:
\begin{enumerate}[(a)]
\item $(\wt{M},\mb{c})$ admits a pure spinor $\chi$  with (maximally isotropic, $n$-dimensional) integrable kernel $\ker \chi$ satisfying the \emph{twistor spinor equation}
\begin{align}\label{eq-twistor}
  \wt{D}_a\chi+\frac{1}{2n}\ga_a \dirac\chi=0 \, , 
\end{align}
where $\dirac=\ga^c\wt{D}_c$ is the Dirac operator and $\ga$ denotes the Clifford multiplication.
\item $(\wt{M},\mb{c})$ admits a (light-like) conformal Killing field $k$ with $k\in\ker \chi$.
\item
The Lie derivative of $\chi$ with respect to the conformal Killing field $k$ is
\begin{align}\label{eq-Liechi}
  \mathcal{L}_k \chi = -\frac12  (n+1) \chi \, .
\end{align}
\item The following integrability condition is satisfied for all $v^r, w^s\in\ker\chi$:
  \begin{align}\label{eq-weylcond}
    \wt{W}_{abcd} v^aw^d & =0 \, .
  \end{align}
\end{enumerate}
\end{thm}

In section~\ref{sec-einstein}, we achieve a complete description of {Einstein metrics} within the induced conformal class in terms of the underlying geometric objects.   In what follows, 
$R \ind{_{DA}^C_B}$ is the curvature tensor of a torsion-free affine connection $D_A$ and
$W \ind{_{DA}^C_B}$ is the (projectively invariant) totally trace-free part of $R \ind{_{DA}^C_B}$.
That is, we use abstract indices $\wt{\mc{E}}^a \cong T\wt{M}$ on $\wt{M}$ and $\mc{E}^A \cong T{M}$ on ${M}$, and we shall not distinguish between bundles and spaces of sections notationally. Let us emphasize that the theorems below involve certain projectively invariant differential operators, and to formulate the invariance precisely will require the use of density-valued tensor fields. Leaving these details aside for the time being, the results can be stated as follows:
\begin{thm}\label{thm-scales}
\hfill
\begin{enumerate}[(a)]
\item If the affine connection $D$ is Ricci-flat, then the induced Patterson--Walker metric $g$ is Ricci-flat.
\item If the affine connection $D$ admits an {Euler-type} vector field $\xi$ satisfying the projectively invariant equation
\begin{align}\label{eq-euler}
  D_C\xi^A=\frac{1}{n}\de^A_C D_P\xi^P
\end{align}
and the integrability condition $\xi^D W \ind{_{DA}^C_B}  = 0$,
then the induced Patterson--Walker metric $g$ is conformal to a Ricci-flat metric $\si_{\xi}^{-2}g$ off the zero-set of a rescaling function $\si_{\xi}$.
\end{enumerate}

In fact, any Einstein metric in the conformal class $\mb{c}$ can be uniquely
decomposed into two Einstein metrics of such types.
\end{thm}
Part (a) is a well-known fact for Patterson--Walker metrics that was already observed in \cite{patterson-walker, derdzinski}, and which we recover. To our knowledge, the construction of Ricci--flat Einstein metrics of part (b) is new, as is the decomposition result for general Einstein metrics. The decomposition of general Einstein metrics in $\mb{c}$ can be understood explicitly:
if the Patterson--Walker metric $g$ is conformal to an Einstein metric $\si^{-2}g$, then there is a canonical decomposition 
  \begin{align*}
    \si=\si_++\si_-
  \end{align*}
such that both $g_+=\si_+^{-2}g$ and $g_-=\si_-^{-2}g$ are Ricci-flat off the respective zero-sets of $\si_\pm$.
Further, there is a Ricci-flat affine connection $D_-$ projectively related to $D$,
which induces the Ricci-flat Patterson--Walker metric $g_-$, and
an Euler-type vector field $\xi$ for $D$ satisfying \eqref{eq-euler} and the integrability condition $\xi^D W \ind{_{DA}^C_B}  = 0$ such that $g_+=\si_{\xi}^{-2}g$.

Finally, in section \ref{sec-infsym} we study the Riemannian and conformal symmetries of the induced Patterson--Walker metric and present their complete description in terms of affine and projective properties of $D$ and of $\mb{p}$, respectively.
Since the construction of the conformal structure $\mb{c}$ on $\wt{M}=T^*M(2)$ is natural, symmetries of the projective structure $\mb{p}$ give rise to conformal symmetries (i.e.\ conformal Killing fields) of $\mb{c}$.
In fact, we can completely and explicitly understand the space of conformal Killing fields of $\mb{c}$ in terms of solutions to projectively invariant equations:

\begin{thm}\label{thm-confkillingfields}
\hfill
  \begin{enumerate}[(a)]
  \item Any infinitesimal symmetry $v^A$ of the projective structure $\mb{p}$ induces a conformal Killing field $\wt{v}^a_0$ of $\mb{c}$.
\item Any skew-symmetric bivector $w^{AB}$ satisfying the projectively invariant equation
  \begin{align}
    D_C w^{AB}=-\frac{2}{n-1}\,\de_C^{[A}D_P w^{B]P}
  \end{align}
    and the integrability condition
$w^{B(A} W\ind{_{B(C}^{D)}_{E)}}=0$
  induces a conformal Killing field $\wt{v}^a_+$ of $\mb{c}$.
  \item Any Killing $1$-form, i.e.\ a $1$-form $\al_A$ satisfying $D_{(A}\al_{B)}=0$,
    induces a conformal Killing field $\wt{v}^a_-$ of $\mb{c}$.
  \end{enumerate}
  In fact, any conformal Killing field of $\mb{c}$ can be uniquely decomposed as
a direct sum
$\wt{v}^a_+ +  \wt{v}^a _0 + \wt{v}^a _- + c \, k^a$ of components which correspond to solutions to the respective projective equations and a constant multiple of $k$.
\end{thm}

Likewise, the construction of the Patterson--Walker metric $g$ from a torsion-free affine connection $D$ is natural, hence any symmetry of $D$ gives a symmetry of $g$, i.e.\ a Killing field. 
In fact, we obtain a complete description of the space of Killing fields of $g$ in terms of affine data:

\begin{thm}\label{thm-killingfields}
\hfill
  \begin{enumerate}[(a)]
  \item Any infinitesimal symmetry $v^A$ of the affine connection $D$ induces a Killing field $\wt{v}^a_0$ of $g$.
\item Any parallel bivector $w^{AB}$ for the affine connection $D$, $D_C w^{AB}=0$, which satisfies the integrability condition $w^{B(A} R\ind{_{B(C}^{D)}_{E)}}=0$ induces a Killing field $\wt{v}^a_+$ of $g$.
  \item Any Killing $1$-form $\al_A$, $D_{(A}\al_{B)}=0$, induces a Killing field $\wt{v}_-^a$ of $g$.
  \end{enumerate}
In fact, any Killing field of $g$ can be uniquely decomposed as
a direct sum $\wt{v}^a_+ +  \wt{v}^a _0 + \wt{v}^a _-$ of components which correspond to solutions to the respective affine equations.
\end{thm}

The approach of the present paper is based on an extension of the two-spinor calculus of \cite{penrose-rindler-86} to higher dimensions, already used in \cite{hughston-mason}, and developed more fully in \cite{Taghavi-Chabert2012a,Taghavi-Chabert2013}. We shall set up this spinor calculus in section \ref{sec-construction} and employ it to directly derive relationships between the original projective geometry and the induced conformal structure. A major step, which is particularly tailored for this approach, is our parallelizability result for pure twistor spinors with integrable distributions, Proposition \ref{prop-twistor-spinor2parallel}, upon which Theorem \ref{thm-char} hinges.

Projective and conformal geometries are instances of Cartan geometries, or more specifically, parabolic geometries. The  geometric relationship studied in this article fits into the larger framework of so-called Fefferman-type constructions. These were originally put forward by the authors of \cite{fefferman} and \cite{graham} in their investigation of CR structures. In the present context, the recent article \cite{hsstz-fefferman} takes the same perspective, and includes a characterisation 
result closely related to Theorem \ref{thm-char}. The relation with the treatment set forth herein is briefly described in section \ref{relation}. The spinor-theoretic approach allows a succinct treatment, gives a shorter statement for the characterization of the induced structures than the one presented in \cite{hsstz-fefferman}, and allows us to give explicit descriptions of the Einstein metrics in the induced conformal class of metrics.

\subsubsection*{Funding}
This work was supported by the Austrian Science Fund [J3071-N13 to K. S.]; the Czech Science Foundation [P201/12/G028 to J. \v{S}., GP14-27885P to A. T.-C.]; and the University of Turin to [A. T.-C.].

K. S. was supported by the National Science Centre Poland (NCN) via the POLONEZ grant 2016/23/P/ST1/04148. This project has received funding from the European Union's Horizon 2020 research and innovation programme under the Marie Sk\l odowska-Curie grant agreement No 665778. \includegraphics[width=0.05\textwidth]{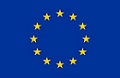}

A. T.-C. was supported by a long-term faculty development grant from the American University of Beirut for his visit to IMPAN, Warsaw, where the final revisions of the article were carried out.

\subsubsection*{Acknowledgments}
The authors express special thanks to Maciej Dunajski for motivating the study of this construction and for a number of enlightening discussions on this and adjacent topics.

\section{Patterson--Walker metrics}\label{sec-Riem-ext}
Riemann extensions of affine connected spaces were first described in
\cite{patterson-walker}.
They are pseudo-Riemannian metrics on the total space of the cotangent bundle $\pi: T^*M\to M$ associated to torsion-free affine connections on $M$ as follows:
An affine connection $D$ determines a horizontal distribution
$H\subset T(T^*M)$ complementary to the vertical distribution $V$ of
the bundle projection $\pi$.
Via the tangent map of $\pi$, the bundle $H$ is isomorphic to $TM$, whilst $V$ is canonically isomorphic to $T^*M$.
\begin{defn}
  The  \emph{Riemann extension} or the \emph{Patterson--Walker metric}
  associated to a torsion-free affine connection $D$ on $M$ is the split-signature metric $g$ on $\wt{M}:=T^* M$
  fully determined by the following conditions:
  \begin{enumerate}[(a)]
    \item both $V$ and $H$ are isotropic with respect to $g$,
    \item the value of $g$ with one entry from $V$ and another entry from $H$ is given by the natural pairing between $V\cong T^*M$ and $H\cong TM$.
      \end{enumerate}
\end{defn}
It follows that $V$ is parallel with respect to the Levi-Civita
connection of the constructed metric. Hence Riemann extensions are special cases of pseudo-Riemannian manifolds admitting a parallel isotropic distribution known as \emph{Walker manifolds} or \emph{Walker metrics}.

\smallskip
We can give local coordinate expression for these Riemann extensions. 
Let us introduce local coordinates $\{ x^A \}$ on $M$ and fibre coordinates $\{ p_A \}$  so that $\theta=p_A \d x^A$ is the tautological $1$-form on  $\wt{M}$. Here, indices run from $1$ to $n$, but we shall view them as abstract indices.
Let further $\Gamma \ind{_A^C_B}=\Gamma \ind{_{(A}^C_{B)}}$ be the Christoffel symbols of a torsion-free affine connection $D_A$ on $M$. The horizontal distribution $H$ associated to the affine connection $D_A$ is spanned by
\begin{align}\label{eq-HD-coord}
\parderv{}{x^A} + \Gamma \ind{_A^C_B} \, p_C \, \parderv{}{p_B} \, .
\end{align}
Defining $\alpha \odot \beta := \frac{1}{2} \left( \alpha \otimes \beta + \beta \otimes \alpha  \right)$ for any $1$-forms $\alpha$ and $\beta$, we can write the Patterson--Walker metric explicitly as
\begin{align}\label{pattersonwalker}
g &=  2 \, \d x^A \odot \d p_A -2 \, \Gamma \ind{_A^C_B} \, p_C  \, \d x^A \odot \d x^B \, ,
\end{align}
from which it is clear that both $V=\left\langle\parderv{}{p_B}\right\rangle$ and $H$ spanned by \eqref{eq-HD-coord} are indeed isotropic with respect to \eqref{pattersonwalker}.

Being oriented, the cotangent bundle $T^*M$ equipped with the Patterson--Walker metric has structure group  $\SO(n,n)$. Further, following \cite{gualtieri,karoubi}, since $T(T^*M) \cong TM \oplus T^*M$, the manifold $T^*M$ is endowed with a spin structure. Since $V$ and $H$ are totally isotropic and dual to each other via the metric, we can associate to them a pair of pure spinors defined up to scale. 
These spinors will allow us to construct projections from $T \widetilde{M}$ to $V$ and $H$.
With a slight abuse of notation to be clarified subsequently, it will be convenient to employ abstract index notation on spinor fields (see \cite{penrose-rindler-84}):
sections of the irreducible spinor bundles $\widetilde{S}_+$ and $\widetilde{S}_-$ will be adorned with primed and unprimed upper-case Roman indices, i.e.\ $\alpha^{A'} \in \widetilde{S}_+$ and $\beta^A \in \widetilde{S}_-$, and similarly for dual spinor bundles, $\kappa_{A'} \in \widetilde{S}_+^*$ and $\lambda_A \in \widetilde{S}_-^*$. In particular, the Clifford algebra of $(T \widetilde{M},g)$ is generated by the $\gamma$-matrices $\gamma \ind{_a^{B'}_A}$ and $\gamma \ind{_a^B_{A'}}$, which satisfy
\begin{align*}
\gamma \ind{_{(a}^{A'}_C} \gamma \ind{_{b)}^C_{B'}} & = - g_{ab} \wt{\delta}^{A'}_{B'} \, , & \gamma \ind{_{(a}^A_{C'}} \gamma \ind{_{b)}^{C'}_B} & = - g_{ab} \wt{\delta}_A^B \, ,
\end{align*}
where $\wt{\delta}^{A'}_{B'}$ and $\wt{\delta}^A_B$ are the identity elements on $\wt{S}_+$ and $\wt{S}_-$ respectively.

Let $\chi^{A'} \in \widetilde{S}_+$ be a spinor field annihilating $V$, and define a linear map
\begin{align*}
  \chi_a^A := \gamma \ind{_a^A_{B'}} \chi^{B'} : T \widetilde{M} \rightarrow \widetilde{S}_- \, .
\end{align*}
Then $V = \ker \chi_a^A$ since $\chi^{A'}$ is pure.
Similarly, let $\check{\eta}_{A'} \in \widetilde{S}_+^*$ be a spinor field annihilating $H$ so that $\chi^{A'}$ and $\check{\eta}_{A'}$ are dual, and chosen such that $\check{\eta}_{A'} \chi^{A'} = - \frac{1}{2}$. Defining
\begin{align*}
  {\check{\eta}}_{aA} := {\check{\eta}}_{B'} \gamma \ind{_a^{B'}_A} : T \widetilde{M} \rightarrow \widetilde{S}_-^* \, ,
\end{align*}
we then have $H = \ker {\check{\eta}}_{aA}$ since $\check{\eta}_{A'}$ is pure.

Therefore, we can identify $H$ with the image of $\chi_a^A$, and $V$ with the image of $\check{\eta}_{aA}$. In this situation the upper case Roman index refers to an $n$-dimensional representation. Viewed as projections, the spinors satisfy \cite{Taghavi-Chabert2012a}
\begin{align}\label{xinu}
\chi_a^A \chi^{aB} & = 0 \, ,  & \check{\eta}^a_A \check{\eta}_{aB} & = 0 \, , & \chi_a^A \check{\eta}^a_B & = \delta_A^B \, ,
\end{align}
where $\delta_A^B$ is the identity on $\im \chi_a^A$. In sum, we have a splitting
\begin{align*}
T \wt{M} & = V  \oplus H \cong \im \check{\eta}_{aA}  \oplus  \im \chi_a^A  \cong  \ker \chi_a^A  \oplus  \ker \check{\eta}_{aA} 
\end{align*}
where $H \cong V^*$, and for any $\wt{v}^a \in \Vertical, \wt{w}^a \in H$, we can write
\begin{align*}
  \wt{v}^a & = \wt{\alpha}_A \chi^{aA} \, , & \mbox{for some $\wt{\alpha}_A \in \im {\check{\eta}}_{aA}$,} \\
  \wt{w}^a & = \wt{\beta}^A  {\check{\eta}}^a_A  \, , & \mbox{for some $\wt{\beta}^A \in \im {\chi}^{aA}$.} 
\end{align*}

There is the freedom in rescaling both $\chi^{A'}$ and $\check{\eta}_{A'}$ such that $\chi^{A'} \check{\eta}_{A'} = -\frac{1}{2}$, which will be fixed by the following consideration. If the torsion-free affine connection $D$ preserves in addition a volume form on $M$,  then the connection $D$ is said to be \emph{special}, and all our affine connections will have this property.
This means that we can always choose our coordinates $\{ x^A \}$ such that the preserved volume form is given by $\d x^1 \wedge \ldots \wedge \d x^n$, up to constant multiple, and thus,  the Christoffel symbols satisfy $\Gamma \ind{_A^C_C}=0$. Henceforth, we denote by $\widetilde{D}_a$ the Levi-Civita connection of the Patterson--Walker metric \eqref{pattersonwalker} on $\wt{M}$ induced by a special torsion-free affine connection $D_A$ on $M$. Since $V=\left\langle\parderv{}{p_B}\right\rangle$ and $H$ is spanned by \eqref{eq-HD-coord}, we can choose
$\chi^{A'}$ and $\check{\eta}_{A'}$ such that
\begin{align}\label{eq-for-Josef}
\chi^{aA} \wt{D}_a & = \parderv{}{p_A} \, , & \check{\eta}^a_A \wt{D}_a & = \parderv{}{x^A} + \Gamma \ind{_A^C_B} \, p_C \, \parderv{}{p_B} \, ,
\end{align}
 and the non-trivial commutation relations
\begin{align*}
[ \chi^{aA} \wt{D}_a , \check{\eta}^b_B \wt{D}_b ] & = \Gamma \ind{_B^A_C} \chi^{cC} \wt{D}_c \, , &
[ \check{\eta}^a_A \wt{D}_a , \check{\eta}^b_B \wt{D}_b ] & = R \ind{_{AB}^C_D} p_C \chi^{cD} \wt{D}_c \, ,
\end{align*}
are satisfied. Here we use the convention $R \ind{_{AB}^C_D} v^D = 2 \, D_{[A} D_{B]} v^C$ for the  curvature tensor
$R \ind{_{BC}^D_A}$ of $D_A$. 
We can immediately see that $H$ is integrable if and only if $D_A$ is flat. We then obtain the Christoffel symbols $\widetilde{\Gamma} \ind{_a^c_b}$ of the connection $\widetilde{D}_a$
\begin{align*}
 \widetilde{\Gamma} \ind{_{abc}} & = 2 \, {\chi} \ind*{_a^A} {\check{\eta}} \ind*{_{[b}_B} {\chi} \ind*{_{c]}^C} \Gamma \ind{_{A}^B_C} + {\chi} \ind*{_a^A} {\chi} \ind*{_{b}^B} {\chi} \ind*{_c^C} R \ind{_{BC}^D_A} p_D \, .
\end{align*}
In particular, using \eqref{xinu} and the fact $\Gamma \ind{_A^B_C}$ is trace-free, we immediately see that the spinor $\chi^{A'}$ determined by \eqref{eq-for-Josef} is parallel. Writing $\wt{v}^a = \wt{v}^A \check{\eta}^a_A + \wt{\alpha}_A \chi^{aA}$, we have
\begin{align*}
\left( \wt{D}_a \wt{v}^b \right) \check{\eta}^a_A \chi_b^B & = \left( \parderv{}{x^A} + \Gamma \ind{_A^C_D} \, p_C \, \parderv{}{p_D} \right) \wt{v}^B + \Gamma \ind{_A^B_C} \wt{v}^C  \, , \\
\left( \wt{D}_a \wt{v}^b \right) \check{\eta}^a_A \check{\eta}_{bB} & = \left( \parderv{}{x^A} + \Gamma \ind{_A^C_D} \, p_C \, \parderv{}{p_D} \right) \wt{\alpha}_B - \Gamma \ind{_A^C_B} \wt{\alpha}_C  - \wt{v}^C R \ind{_{CB}^D_A} p_D \, , \\
\left( \wt{D}_a \wt{v}^b \right) \chi^{aA} \check{\eta}_{bB} & = \parderv{}{p_A} \wt{\alpha}_B \, , \\
\left( \wt{D}_a \wt{v}^b \right) \chi^{aA} \chi_b^B & = \parderv{}{p_A} \wt{v}^B \, .
\end{align*}
In particular, if $\wt{v}^B = v^B (x)$ and $\wt{\alpha}_B = \alpha_B  (x)$ do not depend on $p_A$, then
\begin{align}\label{eq-Dv-no-p}
 \wt{D}_a \wt{v}^b & = \left( D_A v^B \right) \chi_a^A \check{\eta}^b_B
 + \left( D_A \alpha_B  - v^C R \ind{_{CB}^D_A} p_D \right) \chi_a^A \chi^{bB} \, .
\end{align}
Next, the Riemann tensor can be computed to be
\begin{multline}\label{eq-Riem2AffCurv}
  \widetilde{R} \ind{_{abcd}} = 2 \left( \chi_a^A \chi_b^B \check{\eta}_{[c C} \chi_{d]}^D + \chi_c^A \chi_d^B \check{\eta}_{[a C} \chi_{b]}^D  \right) R \ind{_{AB}^C_D} \\
  + 2 \, \chi_{[a}^A \chi_{b]}^B \chi_c^C \chi_d^D D_A R \ind{_{CD}^E_B} p_E \, ,
\end{multline}
from which we deduce that
\begin{align}
\label{eq-Weyl-cond-PW}
  \widetilde{R}_{abcd} v^aw^d & =0  & & \mbox{for all $v^a, w^a \in V$.}
\end{align}
We have a distinguished vector field $k$ and a $2$-form $\mu$, defined by 
\begin{align}
k & := 2 \, p_A \parderv{}{p_A} \, , \label{eq-Special-K} \\
 \mu & := 2 \, \d p_A \wedge \d x^A \, . \label{eq-symp}
\end{align}
Here, we follow the convention $\alpha \wedge \beta = \frac{1}{2} \left(\alpha \otimes \beta - \beta \otimes \alpha\right)$ for any $1$-forms $\alpha$ and $\beta$.  As a $1$-form, $k_a$ is twice the tautological one-form $\theta_a$ on $T^* M$. As a skew-symmetric endomorphism, $\mu \ind{^a_b}$  acts as the identity on $H$ and as minus the identity on $V$:
\begin{align}\label{eq-mu-eigen}
\mu \ind{^a_b} \check{\eta}^b_B & = \check{\eta}^a_B \, , & \mu \ind{^a_b} \chi^{bB} & = - \chi^{aB} \, .
\end{align}
It is then straightforward to check that $k$ satisfies the \emph{conformal Killing field equation}
\begin{align}
\label{eq-CKf2homothety}
\widetilde{D}_a k_b - \mu_{ab} - g_{ab} & = 0,
\end{align}
and in particular, $k^a$ is a light-like vertical homothety, $\mc{L}_k g=2\,g$.

Now, Patterson--Walker metrics can be locally characterized as follows:
\begin{prop}\label{prop-desc-PW}
Let $(\wt{M},g)$ be a spin structure of split signature $(n,n)$ admitting a parallel pure spinor $\chi$ with integrable associated distribution $V$, and a homothety $k$ tangent to $V$ such that 
\eqref{eq-CKf2homothety} holds. Suppose further that the Riemann tensor satisfies
\eqref{eq-Weyl-cond-PW}.

Then, in a neighborhood of any point of $\wt{M}$, there exist coordinates $\{x^A,p_A\}$ such that the metric $g$ takes the form
\eqref{pattersonwalker} where $\Gamma \ind{_A^C_B}$ are the Christoffel symbols for a special torsion-free affine connection $D$ on the leaf space of $V$.
In particular, $(\wt{M},g)$ is the Riemannian extension associated to $D$.
\end{prop}

\begin{proof}
In a neighborhood of any point of $\widetilde{M}$, there exist coordinates $\{x^A, p_A\}$ such that the metric takes the form \cite{bryant-pseudo,kath-pseudo}
\begin{align}\label{eq-parallel_spinor_metric}
 g & = 2 \, \d p_A \odot \d x^A - 2 \, \Theta \ind{_{AB}} \d x^B \odot \d x^A  \, , 
\end{align}
where the distribution $\Vertical$ is spanned by the vector fields $\parderv{}{p_A}$ and $\{ x^A \}$ are coordinates on the leaf space $M$, and the functions $\Theta \ind{_{AB}} = \Theta \ind{_{(AB)}} (x,p)$ satisfy the differential conditions
\begin{align}\label{eq-xtra-spinor}
\parderv{}{p_B} \Theta \ind{_{BA}} = 0 \, .
\end{align}
Since $k$ is a homothety tangent to $\Vertical$, we can write
\begin{align*}
k & = k_A \parderv{}{p_A} \, , & g(k,-) & = k_A \d x^A \, ,
\end{align*}
for some functions $k_A$. The exterior derivative of this $1$-form is given by
\begin{align*}
\mu  = \parderv{}{x^A} k_B \d x^A \wedge \d x^B + \parderv{}{p_A} k_B \d p_A \wedge \d x^B \, . 
\end{align*}
This gives
\begin{align*}
\mu  \left( \parderv{}{p_A} , - \right) & = \frac{1}{2} \parderv{}{p_A} k_B \d x^B \, .
\end{align*}
Since $\chi$ is parallel, differentiating $k^a \chi_a^A = 0$ yields
\begin{align*}
\mu_{ab} \gamma^b \chi + \gamma_a \chi = 0,
\end{align*} 
according to \eqref{eq-CKf2homothety}.
This means that $\mu$, as an endomorphism of $T\wt{M}$, acts by minus the identity on $V$.
Hence $\frac{1}{2} \parderv{}{p_A} k_B = \delta_A^B$, i.e.\ $k_B = 2 \, p_B + \phi_B$ for some functions $\phi_B$ of $x^A$. We can perform a change of the coordinates $p_A$ to eliminate the functions $\phi_B$ in $k_B$ while preserving the form of the metric. At this stage, we have the following local coordinate forms for the homothety $k^a$, its associated $1$-form $k_a$, and its exterior derivative $\mu_{ab}$:
\begin{align*}
k & = 2 \, p_A \parderv{}{p_A} \, , & g(k,-) & = 2 \, p_A \d x^A  \, \\
\mu & = 2 \, \d p_A \wedge \d x^A   \, .
\end{align*}
Now, $k$ is a homothety satisfying $\mc{L}_k g_{ab} = 2 \, g_{ab}$, and the
equivalent condition on $\Th_{AB}$ is
\begin{align} 
p_C \parderv{}{p_C} \Theta_{AB} & = \Theta_{AB} \, . \label{eq-homogeneous_Theta}
\end{align}
This says that $\Theta \ind{_A_B}$ is homogeneous of degree $1$ in $p_A$.
On the other hand, the curvature condition \eqref{eq-Weyl-cond-PW} is equivalent to
\begin{align}
 \parderv{^2}{p_B \partial p_D} \Theta \ind{_A_C} & = 0 \, , \label{eq-linear_Theta}
\end{align}
which tells us that $\Theta \ind{_A_C}$ is linear in $p_A$.

Putting things together we see that, given the metric \eqref{eq-parallel_spinor_metric}, the conditions \eqref{eq-xtra-spinor}, \eqref{eq-homogeneous_Theta} and \eqref{eq-linear_Theta} are satisfied if and only if
$\Theta_{AB}$ takes the form 
\begin{align}\label{eq-Theta_projective}
 \Theta \ind{_{AB}} = \Gamma \ind{_{A}^C_B} \, p_C \, ,
\end{align}
for some $\Gamma \ind{_{A}^C_B} = \Gamma \ind{_{(A}^C_{B)}} (x)$, 
which is moreover trace-free by virtue of \eqref{eq-xtra-spinor}.

The condition \eqref{eq-Weyl-cond-PW}
is the obstruction for the Levi-Civita connection to descend to an affine connection on $M$, cf. \cite{afifi,derdzinski}. We have therefore recovered the Patterson--Walker metric \eqref{pattersonwalker}, and $\Gamma \ind{_A^C_B}$ can be identified with the Christoffel symbols of a special affine connection $D$ on the leaf space of $\Vertical$.
\end{proof}

\section{Conformal Patterson--Walker metrics}\label{sec-construction}
We now deal with a projective-to-conformal analog of the construction from the previous section.

\subsection{Calculus for projective geometry}
As before, we shall use upper case Roman abstract indices as in \cite{penrose-rindler-84} for tensors on $M$. For instance, $\al_A \in \ce_A$ denotes a $1$-form on $M$, $v^{AB}\in\ce^{[AB]}$ denotes a  bivector on $M$. This convention should not be confused with unprimed spinor indices. By and large, we follow the treatment given in \cite{eastwood-notes,eastwood-notes-conformal, thomass}.

Two torsion-free affine connections $D_A$ and $\widehat{D}_A$ are in a given projective class $\mb{p}$ if and only if for any $\xi^A \in \mc{E}^A$,
\begin{align}\label{eq-Dhat-D}
\widehat D_A\xi^B & =D_A\xi^B+ Q \ind{_{AC}^B} \xi^C \, , & Q \ind{_{AB}^C} = 2 \, \delta_{(A}^C \Upsilon_{B)} \, ,
\end{align}
for some $1$-form $\Upsilon_A$. Similar formulae can be obtained on $1$-forms and tensors by means of the Leibniz rule.

We shall assume $M$ to be oriented. Let us fix a volume form $\varepsilon_{A_1 \ldots A_n} \in \mc{E}_{[A_1 \ldots A_n]}$. Then, by \eqref{eq-Dhat-D},  for any two affine connections $D_A$ and $\widehat{D}_A$ in $\mb{p}$, we have
\begin{align}\label{eq-diff-Deps}
\widehat{D}_A \varepsilon_{B_1 \ldots B_n} & = D_A \varepsilon_{B_1 \ldots B_n} - (n+1) \, \Upsilon_A \varepsilon_{B_1 \ldots B_n} \, ,
\end{align}
for some $1$-form $\Upsilon_A$. We can always choose $\varepsilon^{A_1 \ldots A_n} \in \mc{E}^{[A_1 \ldots A_n]}$ such that $\varepsilon_{A_1 \ldots A_n} \varepsilon^{B_1 \ldots B_n} = n! \delta_{[A_1}^{B_1} \ldots \delta_{A_n]}^{B_n}$. In general, $D_A$ does not preserve $\varepsilon_{A_1 \ldots A_n}$ so that if we set
\begin{align}\label{eq-special-Upsilon}
\Upsilon_A := \frac{1}{(n+1)!} \left( D_A \varepsilon_{B_1 \ldots B_n}  \right) \varepsilon^{B_1 \ldots B_n} \, ,
\end{align}
the connection $\widehat{D}_A$ given by \eqref{eq-Dhat-D} or \eqref{eq-diff-Deps} preserves $\varepsilon_{A_1 \ldots A_n}$. Thus, we can always find a special connection, i.e.\ a connection that preserves a given volume form, in the projective class $\mb{p}$, and such a connection can be shown to be unique, cf.\ \cite{bryant-2009} and \cite{dunajski-2016}.

With no loss of generality, we shall henceforth restrict ourselves to special torsion-free affine connections. These enjoy
nice properties. In particular, if $R\ind{_{AB}^C_D}$ is the curvature tensor of a special torsion-free affine connection $D$ with Ricci tensor $\Ric_{AB} := R\ind{_{PA}^P_B}$, then the \emph{Schouten tensor}
\begin{align*}
  \Rho_{AB} := \frac{1}{n-1} \Ric_{AB} \, ,
\end{align*}
is symmetric. Hence $\Rho$ vanishes if and only if $D$ is Ricci-flat.
The projective Weyl curvature and the Cotton tensor are defined respectively by 
\begin{align}	\label{formulaCprojective}
  W\ind{_{AB}^C_D} =R\ind{_{AB}^C_D} +\Rho_{AD}\de\ind*{^C_{B}}- 
     \Rho_{BD}\de\ind*{^C_{A}}, \quad
  Y_{CAB}=2D_{[A}\Rho_{B]C}. 
\end{align}
The connection $D$ is called \emph{projectively flat} if it is projectively equivalent to a flat affine connection.
For manifolds of dimension $n=2$, the Weyl curvature vanishes identically and the only obstruction to projective flatness is the Cotton tensor $Y$.
For $n\geq 3$ projective flatness is equivalent to the vanishing of the Weyl curvature $W$.

By \eqref{eq-diff-Deps}, any two volume forms $\varepsilon$ and $\hat{\varepsilon}$ related by $\hat{\varepsilon} = \ee^{(n+1) \phi} \varepsilon$ correspond to two special torsion-free affine connections $D$ and $\widehat{D}$ differing by the $1$-form $\Upsilon_A = D_A \phi$. We note that under such a projective change, the Rho tensor transforms according to
\begin{align}
    \widehat \Rho_{AB} & =\Rho_{AB} + \Ups_A \Ups_B - D_A \Ups_B \, ,
\end{align}
so that the Schouten $\widehat{\Rho}_{AB}$ associated to $\widehat{D}_A$ remains symmetric.

We therefore have a special subclass of torsion-free affine connections of $\mb{p}$, projectively related by exact $1$-forms, and thus parametrized by smooth functions on $M$. We can conveniently define the \emph{density bundle of projective weight $w$} as $\mc{E}(w):=\left(\wedge^n T M\right)^{-\frac{w}{n+1}}$ on $M$, where $\dim M=n$. 
We will refer to everywhere positive sections of  $\mc{E}(1)$ as \emph{projective scales}. Any projective scale $\sigma$, say, determines a special torsion-free affine connection $D_A$ in $\mb{p}$, which extends to an affine connection, also denoted $D_A$, on $\mc{E}(w)$, and for which $D_A \sigma =0$. For any two torsion-free affine connections in $\mb{p}$, we have
\begin{align}\label{eq-diff-Ddens}
\widehat D_A f & =D_A f + w \Ups_A f \, , & f \in \mc{E}(w) \, ,
\end{align}
An oriented projective structure determines a distinguished section $\bm{\upvarepsilon}_{A_1 \ldots A_n} \in \mc{E}_{[A_1 \ldots A_n]}(n+1)$, which we shall refer to as the \emph{projective volume form}. Any choice of projective scale $\sigma$ corresponds to a special connection $D$ preserving the volume form $\varepsilon = \sigma^{-(n+1)} \bm{\upvarepsilon}$. Since, for \emph{any} two connections $D$ and $\widehat{D}$ in $\mb{p}$, we have $\widehat{D} \bm{\upvarepsilon} = D \bm{\upvarepsilon}$ by \eqref{eq-diff-Ddens} and \eqref{eq-diff-Deps}, we conclude that $D \bm{\upvarepsilon} = 0$ for any connection $D$ in $\mb{p}$.

\subsection{Calculus for conformal geometry}
As before, we shall use lower case Roman indices for tensors on $\wt{M}$, e.g.\ $g_{ab}\in\wt{\ce}_{(ab)}$ denotes a symmetric $2$-tensor on $\wt{M}$. The reader can refer to \cite{thomass} for more details on conformal geometry and its calculus.

We define the \emph{density bundle of conformal weight $w$} as $\wt{\mc{E}}[w]:=\left(\wedge^{2n} T \wt{M} \right)^{-\frac{w}{2n}}$ on $\wt{M}$, where $\dim\wt{M}=2n$. 
We will refer to everywhere positive sections of $\wt{\mc{E}}[1]$ as \emph{conformal scales}. The Levi-Civita connection extends to an affine connection on $\wt{\mc{E}}[w]$. The conformal structure can be equivalently seen as a density-valued metric $\mb{g}_{ab} \in \wt{\mc{E}}_{(ab)}[2]=\bigodot^2 T^*\wt{M}\otimes\wt{\mc{E}}[2]$ referred to as the \emph{conformal metric} on $\wt{M}$. Any conformal scale $\tau \in \wt{\mc{E}}[1]$ determines a metric $g_{ab} = \tau^{-2} \mb{g}_{ab}$ in $\mb{c}$. The associated Levi-Civita connection $\wt{D}_a$ preserves $g_{ab}$, $\mb{g}_{ab}$ and $\tau$. The conformal metric allows us to identify $T \wt{M}$ with $T^* \wt{M}[2]$. Similarly, one can identify $\wt{S}_\pm$ with $\wt{S}^*_\pm[1]$ when $n$ is even, and with $\wt{S}^*_\mp[1]$ when $n$ is odd, by means of weighted spin bilinear forms.

For a (pseudo-)Riemannian metric $g$, the Schouten tensor $\wt{\Rho}$ is given by 
\begin{align*}
  \wt{\Rho}_{ab} =\frac{1}{2n-2}\left(\wt{\mr{Ric}}_{ab} -\frac{\wt{\mr{Sc}}}{2(2n-1)}\,g_{ab} \right),
\end{align*}
where $\wt{\mr{Ric}}$ and $\wt{\mr{Sc}}$ are the Ricci and scalar curvature of $g$ respectively.
Since $\wt{\Rho}$ is a trace modification of $\wt{\mr{Ric}}$, the Schouten tensor vanishes if and only if $g$ is Ricci-flat. 
The conformal Weyl curvature and the Cotton tensors of $g$ are defined respectively by 
\begin{align*}
  \wt{W} \ind{_{ab}^c_d}  & =\wt{R} \ind{_{ab}^c_d} - 2 \, \de_{[a}^c\wt{\Rho}_{b]d}+2g_{d[a}\wt{\Rho}_{b]}{}^c \, , &
  \wt{Y}_{cab} & = 2\wt{D}_{[a} \wt{\Rho}_{b]c}.
\end{align*}
The metric $g$ is called \emph{conformally flat} if it can be (locally) rescaled to a flat metric.
For manifolds of dimension $2n\geq 4$ conformal flatness is equivalent to the vanishing of the Weyl curvature $\wt{W}$. The transformation rules for Levi-Civita connections and Schouten tensors under conformal changes can be given explicitly, see e.g.\ \cite{thomass}.

\subsection{Conformal extensions of projective structures}
The Riemann extension of an affine connected space can be adapted to
weighted cotangent bundles 
$
T^*M (w) =T^*M\otimes\mc{E}(w) $.
The only difference in the weighted case is that a choice of torsion-free affine connection
$D$ gives rise to a weighted metric.
This means that the natural pairing between $H\cong TM$ and
$V\cong T^*M(w)$ defines a symmetric bilinear form on the tangent bundle of $T^*M (w) $ with values in
$\pi^*\mc E(w)$, the pull-back of the line bundle over $M$ with respect to the natural projection
$\pi: T^*M (w) \to M$. A special connection $D$ yields a trivialization of $\mc{E}(w)$, and thus the pairing can be regarded as $\mbb{R}$-valued. In particular, $D$ defines a Patterson--Walker metric on $T^*M (w) $.

 We shall denote by $\bm{\uptheta}$ the (weighted) tautological $1$-form on $T^*M (w) $.  This bundle is trivialized by any choice of projective scales. Let $\sigma$ and $\hat{\sigma}$ be two such scales related by $\hat{\sigma} = \ee^{-\phi} \sigma$ for some smooth function $\phi$. Then, $\theta : = \sigma^{-w} \bm{\uptheta}$ and $\hat{\theta} : = \hat{\sigma}^{-w} \bm{\uptheta}$ are two (tautological) $1$-forms related by $\hat{\theta} = \ee^{w \phi} \theta$. In both cases, there exists canonical coordinates $\{ x^A , p_A\}$ and $\{ x^A , \hat{p}_A \}$ in which $\theta = p_A \d x^A$ and 
 $\hat{\theta} =  \hat{p}_A \d x^A$. Thus, a projective change induces the change of canonical fiber coordinates $p_A \mapsto \hat p_A= \ee^{w \phi}p_A$.

Let  $D_A$ and $\widehat{D}_A \in \mb{p}$ be the special affine connections in $\mb{p}$ associated to $\sigma$ and $\hat{\sigma}$ respectively, so that $\widehat{D}_A$ differs from $D_A$ via \eqref{eq-Dhat-D} by $\Upsilon_A = D_A \phi$.  This means that the Christoffel symbols of $D_A$ and $\widehat{D}_A$ are related by
\begin{align*}
    \widehat\Ga \ind{_A^C_B} & = \Ga \ind{_A^C_B} +\de_A^C\Ups_B+\de_B^C\Ups_A \, .
\end{align*}
A straightforward computation then gives
\begin{multline}\label{eq-Hhat-H}
\d \hat{p}_A - \widehat{\Gamma} \ind{_A^C_B} \hat{p}_C \d x^B = \ee^{w \phi} \left( \d p_A - \Gamma \ind{_A^C_B} p_C \d x^B \right)  \\
+ \ee^{w \phi} \left( (w - 1) p_A \Upsilon_B - p_B \Upsilon_A \right) \d x^B \, ,
\end{multline}
so that using \eqref{pattersonwalker} yields 
\begin{align*}
    \hat g= \ee^{w \phi}(g+2\,(w-2)p_B\Ups_A\d x^B\odot\d x^A) \, ,
\end{align*}
As a consequence, we immediately conclude:

\begin{prop}\label{prop-weight}
  Let $D$ and $\widehat{D}$ be projectively equivalent special torsion-free affine connections
  on $M$ and let $g$ and
  $\hat g$ be the associated Patterson--Walker metrics on $T^*M (w) $.
  Then $g$ and $\hat g$ are conformally equivalent if and only if $w=2$.
\end{prop}
Setting $\wt{M}:=T^*M(2)$ we have thus obtained the notion of the {conformal extension} $(\wt{M},\mb{c})$ of a projective structure $(M,\mb{p})$:
\begin{defn}\label{def-cPW}
  The  \emph{conformal extension} or the \emph{conformal Patterson--Walker metric}
  associated to an oriented projective structure $\mb{p}$ on $M$ is the split-signature conformal structure $\mb{c}$ on $\wt{M}=T^* M(2)$ represented by the Patterson--Walker metric of a special torsion-free affine connection $D\in\mb{p}$.
\end{defn}

\begin{rema}\label{Th-proj-inv}
A slightly different construction, which was first introduced in \cite{dunajski-tod} when $n=2$, involves the so-called \emph{Thomas projective parameters}.  In dimension $n$, these are defined by \cite{eisenhart,thomas-TY-1934} 
\begin{align}\label{eq-Th-proj-inv}
  \Pi \ind{_A^C_B} & := \Gamma \ind{_A^C_B} - \frac{2}{n+1} \delta_{(A}^C \Gamma \ind{_{B)}^D_D} \, ,
\end{align}
where $\Gamma \ind{_A^C_B}$ are the Christoffel symbols of any affine connection in $\mb{p}$ with respect to some coordinate system $\{ x^A \}$. In fact, the $\Pi \ind{_A^C_B}$ do not depend on the choice of connection in $\mb{p}$, and are thus a set of projectively invariant functions. However, the $\Pi \ind{_A^C_B}$ depend on the choice of coordinates $\{ x^A \}$ in the sense that they do not transform as Christoffel symbols, let alone as a tensor in general. Consider a general coordinate transformation $x^A \mapsto y^A$ on $M$ with Jacobian $J^A_B := \parderv{y^A}{x^B}$, and set $\phi := \frac{1}{n+1} \log \left( \det J^A_B \right)$. Then, we have \cite{eisenhart} 
\begin{align}\label{eq-Pi-Pi'}
\Pi \ind{_A^D_B} J_D^C & = {\Pi'} \ind{_D^C_E} J_A^D J_B^E + \parderv{}{x^A} J_B^C - 2 J_{(A}^C \parderv{\phi }{x^{B)}}  \, ,
\end{align}
where ${\Pi'} \ind{_A^C_B}$ are the Thomas projective parameters defined by the Christoffel symbols with respect to $\{y^A\}$.

The coordinate systems $\{ x^A \}$ and $\{ y^A \}$ define volume forms $\varepsilon := \d x^1 \wedge \ldots \wedge \d x^n$ and $\hat{\varepsilon} := \d y^1 \wedge \ldots \wedge \d y^n$, respectively, preserved by special connections $D_A$ and $\widehat{D}_A$ in $\mb{p}$ respectively. These are projectively related by $\Upsilon_A = D_A \phi$. We therefore have an induced change of canonical fiber coordinates on $T^*M(w)$ given by $p_A \mapsto q_A := \ee^{w \phi} p_B (J^{-1})^B_A$. Define two metrics on the open subset of $T^*M(w) $ over the overlap of the charts of $\{ x^A \}$ and $\{ y^A \}$ by
\begin{align}\label{dunajskitod}
g & :=  2 \, \d x^A \odot \d p_A -2 \, \Pi \ind{_A^C_B} \, p_C  \, \d x^A \odot \d x^B \, , \\
\hat{g} & :=  2 \, \d y^A \odot \d q_A -2 \, {\Pi'} \ind{_A^C_B} \, q_C  \, \d y^A \odot \d y^B \, . \nonumber
\end{align}
Then,  using \eqref{eq-Pi-Pi'}, one can immediately check that
\begin{align*}
\hat{g} = \ee^{w \phi} \left( g + 2 (w-2) q_A \parderv{\phi}{y^B} \d y^A \odot \d y^B \right) \, .
\end{align*}
In particular, $g$ and $\hat g$ are conformally equivalent if and only if $w=2$. We have therefore constructed a conformal class  of metrics of the form \eqref{dunajskitod} on $\wt{M}=T^* M(2)$ from the projective class $\mb{p}$ on $M$: a metric in the conformal class corresponds to the Thomas projective parameters representing $\mb{p}$ in a given coordinate system $\{ x^A \}$, up to coordinate transformations that preserve the volume form $\d x^1 \wedge \ldots \wedge \d x^n$. Different Thomas projective parameters for different coordinate systems yield conformally related metrics.

Finally, with no loss, we can take $\Pi \ind{_A^C_B} = \Gamma \ind{_A^C_B}$ in the definition \eqref{eq-Th-proj-inv}, where $\Gamma \ind{_A^C_B}$ are the Christoffel symbols of the special connection $D_A$ preserving the volume form $\varepsilon$, up to constant multiple, on $(M , \mb{p})$. In this case, the metric \eqref{dunajskitod} can be identified with the Patterson--Walker metric \eqref{pattersonwalker}. As this identification holds for any choice of coordinate system, the conformal class of metrics of the form \eqref{dunajskitod} on $\widetilde{M}$ is none other than the conformal Patterson--Walker metric of Definition \ref{def-cPW}.
\end{rema}

To deal with the conformal class of Patterson--Walker metrics of Definition \ref{def-cPW}, rather than a metric, we shall henceforth view the quantities introduced in section \ref{sec-Riem-ext} as being weighted. In particular, $\gamma \ind{_a^A_{B'}}$ and $\gamma \ind{_a^{A'}_B}$ have conformal weight $1$, $\wt{M}$ being endowed with a conformal spin structure, see also \cite{hsstz-fefferman}. By definition, the conformal Killing field $k^a$ has weight $0$, so that the $1$-form $k_a$ is twice the weighted tautological $1$-form $\bm{\uptheta}_a$ on $\wt{M}$, i.e. $k_a  = 2 \, \bm{\uptheta}_a \in \wt{\mc{E}}_a[2]$. Next, requiring that the spinor $\chi^{A'}$ remain parallel with respect to the Levi-Civita connection of any Patterson--Walker metric, restricts its possible conformal weight. Following the conventions of \cite{penrose-rindler-84,Taghavi-Chabert2012a}, and for convenience, $\chi^{A'}$ will have weight $0$, from which it follows that $\check{\eta}_{A'}$ has weight $0$.

\begin{lem}\label{lemarkable}
Any projective scale $\sigma \in \mc{E}(1)$ lifts to a conformal scale $\wt{\sigma} \in \wt{\mc{E}}[1]$, and thus by extension any section of $\mc{E}(w)$ lifts to a section of $\wt{\mc{E}}[w]$.
Conversely, any section $\wt{\si}$ of $\wt{\mc{E}}[w]$ such that $\chi^{aA}\wt{D}_a \wt{\si}=0$, with respect to any Patterson--Walker metric in $\mb{c}$, descends to a section of $\mc{E}(w)$.

Further, any section $\sigma _{A_1 \ldots A_k}^{B_1 \ldots B_\ell} \in \mc{E}_{A_1 \ldots A_k}^{B_1 \ldots B_\ell}(w)$ gives rise to a section of $\wt{\mc{E}}_{A_1 \ldots A_k}^{B_1 \ldots B_\ell} [w-k+\ell]$. For contravariant tensors, the lifts depend on the choice of special torsion-free affine connection on $\mathbf{p}$.
\end{lem}

\begin{proof}
Proposition \ref{prop-weight} assigns to a special affine connection on $M$, i.e.\ a section of  $\sigma \in \mc{E}(1)$, a Patterson--Walker metric on $\wt{M}$, i.e.\ a section of  $\wt{\sigma} \in \wt{\mc{E}}[1]$. 
This can also be verified by noting that the volume form on $\wt{M}$ induced by $g_{ab} = \wt{\sigma}^{-2} \mb{g}_{ab}$ takes the form
\begin{align*}
\wt{\varepsilon} & = \left( \varepsilon_{A_1 \ldots A_n} \d x^{A_1} \ldots \d x^{A_1} \right) \wedge \left( \varepsilon^{B_1 \ldots B_n} \d p_{B_1} \ldots \d p_{B_n} \right) \, .
\end{align*}
where $\varepsilon_{A_1 \ldots A_n}$ is the volume form determined by $\sigma$, and $\varepsilon^{A_1 \ldots A_n}$ its inverse. Since a special projective change induces a change $\hat{p}_A = \ee^{2 \phi} p_A$ for some function $\phi$, the volume form $\wt{\varepsilon}$ transforms to $\widehat{\wt{\varepsilon}} = \ee^{2n \phi} \wt{\varepsilon}$ as expected. 
The converse statement follows from the fact that the vectors $\chi^{aA}\wt{D}_a$,  for any Patterson--Walker metric in $\mb{c}$, span the vertical distribution.

According to our conventions, we obtain weighted projectors and injectors 
$\chi_a^A \in \wt{\mc{E}}_a^A[1]$ and $\check{\eta}^a_A \in \wt{\mc{E}}^a_A[-1]$. 
Now choosing an affine connection $D \in \mb{p}$, any section $v^A \in \mc{E}^A(w)$ can be canonically lifted 
$\wt{v}^a \in \wt{\mc{E}}^a[w]$. This means in particular that as a spinor field, $v^A = \wt{v}^a \chi_a^A$ gives rise
to a section of $\wt{\mc{E}}^A[w+1]$. Similarly (but independently of the choice of $D \in \mb{p}$), any section 
$\alpha_A \in \mc{E}_A(w)$ gives rise to a section of $\wt{\mc{E}}_A[w-1]$. This generalizes to tensor fields of higher valence.
\end{proof}

\smallskip
Now let $D$ be a special torsion-free affine connection on $M$ and $g$ its Patterson--Walker metric on $\wt{M}$. We can decompose \eqref{eq-Riem2AffCurv} further so as to express the
conformal  Weyl, Schouten and Cotton tensors $\widetilde{W} \ind{_{abcd}}, \widetilde{\Rho} \ind{_{ab}}, \widetilde{Y} \ind{_{cab}}$ of the Patterson--Walker metric $g$ in terms of the projective Weyl, Schouten respectively Cotton tensors $W \ind{_{AB}^D_C}$, $\Rho _{AB}$ and $Y _{ABC}$:
 \begin{align}
 \widetilde{W} \ind{_{abcd}} & = 2 \left( {\chi} \ind*{_a^A} {\chi} \ind*{_b^B} {\chi} \ind*{_{[c}^C} {\check{\eta}} \ind*{_{d]}_D}  + {\chi} \ind*{_c^A} {\chi} \ind*{_d^B} {\chi} \ind*{_{[a}^C} {\check{\eta}} \ind*{_{b]}_D} \right) W \ind{_{AB}^D_C} \nonumber \\
 & \qquad \qquad \qquad \qquad + 2 \, {\chi} \ind*{_{[a}^A} {\chi} \ind*{_{b]}^B} {\chi} \ind*{_{[c}^C} {\chi} \ind*{_{d]}^D} \left( D \ind{_A} W \ind{_{CD}^E_B} p_E + p \ind{_C} Y \ind{_{DAB}} \right) \, , \label{eq-Walker-Weyl}\\
\widetilde{\Rho} \ind{_{ab}} & = {\chi} \ind*{_a^A} {\chi} \ind*{_b^B} \Rho \ind{_{AB}}  \, , \label{eq-Walker-Rho} \\
 \widetilde{Y} \ind{_{cab}} & = {\chi} \ind*{_c^C} {\chi} \ind*{_a^A} {\chi} \ind*{_b^B} Y \ind{_{CAB}}. \label{eq-Walker-Cotton}
\end{align}

\begin{rema}
By direct inspection, we find:
\begin{enumerate}[(a)]
\item
By \eqref{eq-Walker-Weyl}, the induced Patterson--Walker metric is conformally flat if and only if the original affine connection is projectively flat.
\item
By \eqref{eq-Walker-Rho}, the induced Patterson--Walker metric is Ricci-flat if and only if the original affine connection is Ricci-flat.
\item
By \eqref{pattersonwalker} and \eqref{eq-Walker-Rho}, a Patterson--Walker metric is Einstein if and only if it is already Ricci-flat.
\end{enumerate}
\end{rema}

\begin{rema}
In contrast with the projective-to-conformal construction described above, the authors of \cite{dunajski-mettler-projective} canonically associate to a projective structure a split-signature \emph{Einstein metric} with non-zero scalar curvature.
\end{rema}

\section{Characterization of conformal Patterson--Walker metrics}\label{sec-char}
We shall now prove our characterization Theorem \ref{thm-char} which exactly specifies those split-signature conformal spin structures that are associated to a projective structure via the conformal extension in the sense of definition \ref{def-cPW}. For this purpose we start by collecting properties of the induced conformal structures:
\begin{prop}\label{prop-conf-extension}
The conformal extension $(\wt{M},\mb{c})$ associated to an oriented projective structure $(M,\mb{p})$ satisfies all the properties (a)--(d) of Theorem \ref{thm-char}.
\end{prop}
\begin{proof}
  Since $\chi$ is parallel with respect to $\wt{D}$, it trivially satisfies the {twistor spinor equation} \eqref{eq-twistor}.

We have already observed in \eqref{eq-CKf2homothety} that $k\in V$ is a (light-like) conformal Killing field.

The general formula for the 
Lie derivative of $\chi$ with respect to the conformal Killing field $k$ is
  \begin{align}\label{liederivative}
    \mc{L}_k\chi=k^a \wt{D}_a\chi-\frac{1}{4}(\wt{D}_{[a}k_{b]})\ga^a\ga^b\chi-\frac{1}{4n}(\wt{D}_pk^p)\chi.
  \end{align}
Hence it is immediate that $\wt{D}_a \chi=0$, $\wt{D}_ak_b=\mu_{ab}+g_{ab}$
and $\mu_{ab} \chi^{bB} = - \chi_a^{B}$
(according to \eqref{eq-mu-eigen} and \eqref{eq-CKf2homothety})
imply \eqref{eq-Liechi}.

The integrability condition \eqref{eq-weylcond} follows immediately from \eqref{eq-Walker-Weyl}.
\end{proof}

For the converse direction we begin with two technical results which will provide a normal form for structures satisfying the above conformal properties.\footnote{AT-C thanks Andree Lischewski for pointing out an unnecessary curvature condition in the statement of Proposition \ref{prop-twistor-spinor2parallel}, which appeared in an earlier version of \cite{Taghavi-Chabert2012a} (preprint \texttt{arXiv:1212.3595}). See also his analogous result in \cite{lischewski2013}.}
\begin{prop}\label{prop-twistor-spinor2parallel}
 Let ${\chi}$ be a pure real twistor spinor on a conformal pseudo-Riemannian manifold $(\widetilde{M} , \mathbf{c})$ of signature $(n,n)$ with associated totally isotropic $n$-plane distribution $\Vertical$. Suppose $\Vertical$ is integrable. Then locally, there is a conformal subclass of metrics in $\mathbf{c}$ for which ${\chi}$ is parallel, i.e.\ if $g$ is any such metric with Levi-Civita connection $\wt{D}$, $\wt{D} {\chi}=0$. Any two such metrics are related by a conformal factor constant along the leaves of $\Vertical$.
\end{prop}

\begin{proof}
In abstract index notation for spinors we write ${\chi}^{A'}$ and ${\check{\chi}}^A :=  \frac{1}{\sqrt{2}n} (\crd {\chi})^A$. The key idea is to use the transformation rule for ${\check{\chi}}^A$ under a conformal change of metric: For a smooth function $\phi\in\cinf(\wt M)$ and $\hat g=\ee^{2\phi}g$ a rescaled metric, the spinor $\check{\chi}^A$ transforms according to
(see e.g. \cite{branson-spin, mrh-thesis})
\begin{align}\label{eq-twistor-conf_transf}
{\check{\chi}}^A & \mapsto \, {\check{\chi}}^A + \frac{1}{\sqrt{2}} (\widetilde{D}_a \phi) {\chi}\ind{^a^A}.
\end{align}
Thus, to find a conformal scaling for which $\chi^{A'}$ is parallel, we must first show that $\check{\chi}^A$ can be expressed as
 \begin{align}\label{eq-potential}
  {\check{\chi}}^A = \frac{1}{\sqrt{2}} {\chi}\ind{^a^A} \widetilde{D}_a \phi \, ,
 \end{align}
for some smooth function $\phi$. We assume of course that $\check{\chi}^A$ is non-vanishing, for otherwise our spinor was already parallel.

We write the twistor equation on $\chi$ as
 $\widetilde{D}_a \chi^{B'}  = - \frac{1}{\sqrt{2}} \check{\chi}_a^{B'},$
and contracting with ${\chi}^{aA}$ this gives
\begin{align}\label{eq-contract}
 {\chi}^{aA} \widetilde{D}_a \chi^{B'} & = - \frac{1}{\sqrt{2}} {\chi}^{aA} \check{\chi}_a^{B'}.
\end{align}
The condition that $V$ is integrable can be re-expressed as \cite{hughston-mason,Taghavi-Chabert2012a}
\begin{align}\label{eq-geosp}
{\chi}^{aA} \widetilde{D}_a {\chi}^{B'} & = \alpha ^A {\chi}^{B'} \, ,
\end{align}
for some spinor $\alpha^A$, which necessarily lies in the image of $\chi^{aA}$.
The equations \eqref{eq-contract} and \eqref{eq-geosp} together imply
\begin{align*}
-\frac{1}{\sqrt{2}} {\chi}^{aA} \check{\chi}_a^{B'}=\alpha ^A {\chi}^{B'}.
\end{align*}
It is shown in \cite{Taghavi-Chabert2012a}, that since $\chi$ is pure, the last formula implies that
\begin{align}\label{eq-alpha}
  \al^A=\sqrt{2}\check{\chi}^A.
\end{align}
In particular, this implies that $\check{\chi}^A$ also lies in the image of $\chi^{aA}$ and thus
\begin{align}\label{eq-inters}
\chi^{aA} \check{\chi}_a^{B'} & = - 2 \, \check{\chi}^A \chi^{B'}.
\end{align}
By differentiating \eqref{eq-potential} one obtains the integrability conditions
\begin{align}\label{eq-int_cond-potential}
 {\chi} \ind{^a^{[A}} \widetilde{D}_a {\check{\chi}} \ind{^{B]}} & = \alpha \ind{^{[A}} {\check{\chi}}^{B]}
\end{align}
for the existence of $\phi$ (see e.g. \cite{hughston-mason,Taghavi-Chabert2012a}). By \eqref{eq-alpha}, the right-hand-side of \eqref{eq-int_cond-potential} vanishes.
On the other hand, the prolongation of the twistor equation $\widetilde{D}_a {\check{\chi}} ^A = - \frac{1}{\sqrt{2}} \widetilde{\Rho}_{ab} {\chi} \ind{^{bB}}$ leads to the vanishing of the left-hand-side of \eqref{eq-int_cond-potential}. Hence both sides of \eqref{eq-int_cond-potential} are zero, and the integrability conditions are therefore satisfied and we can find a local solution $\phi$ of \eqref{eq-potential}.

Finally, by \eqref{eq-twistor-conf_transf}, adding to $\phi$ a smooth function constant along $V$ yields a metric in $\mb{c}$ conformal related to $\hat{g}$, for which $\chi$ is also parallel. This produces the required conformal subclass of $\mb{c}$.
\end{proof}

\begin{rema}
\hfill
\begin{enumerate}[(a)]
\item The relation between pure twistor spinors and the integrability of their associated distributions is already given in \cite{Taghavi-Chabert2012a}. Similar results are obtained in odd dimensions in \cite{Taghavi-Chabert2013}.
\item
A similar argument is employed in \cite{dunajski-2002} in the four-dimensional case to show the existence of a suitable parallelizing scale.
\item
Formula \eqref{eq-inters} is in fact equivalent to $\check{\chi}^A= \frac{1}{\sqrt{2}n} (\crd {\chi})^A$ being pure with associated $n$-plane distribution intersecting that of $\chi^{A'}$ maximally in an $(n-1)$-dimensional distribution \cite{Taghavi-Chabert2012a}.
\end{enumerate}
\end{rema}


\begin{lem}\label{CKf2homothety}
Let $\chi$ be a parallel pure spinor with associated distribution $\Vertical$ and $k^a$ a conformal Killing field tangent to $\Vertical$ such that $\mc{L}_k \chi = -\frac{1}{2} (n+1) \chi$. 
Then $k^a$ is a homothety satisfying \eqref{eq-CKf2homothety}.
\end{lem}

\begin{proof}
We write the conformal Killing field equation as $\wt{D}_a k_b - \mu_{ab} + g_{ab} \varphi = 0$
with $\mu_{ab}=\wt{D}_{[a}k_{b]}$ and $\varphi=-\frac{1}{2n}\wt{D}^p k_p$. Since $\chi$ is parallel, differentiating $k^a \chi_a^A = 0$ yields
\begin{align}\label{eq-parall-chi1}
\mu_{ab} \gamma^b \chi - \varphi \gamma_a \chi = 0
\end{align} so that
\begin{align}\label{eq-parall-chi2}
\mu_{ab} \gamma^a \gamma^b \chi + 2n \varphi \chi = 0 \, .
\end{align}
On the other hand, $\mc{L}_k \chi = -\frac{1}{2} (n+1) \chi$ now reads as
\begin{align}\label{eq-Lie-chi}
- \frac{1}{4} \mu_{ab} \gamma^a \gamma^b \chi + \frac{1}{2} \varphi \chi = - \frac{1}{2} (n+1) \chi \, ,
\end{align}
where we have used \eqref{liederivative} and the fact that $\chi$ is parallel. Combining \eqref{eq-parall-chi2} and \eqref{eq-Lie-chi} yields $\varphi = -1$, hence \eqref{eq-CKf2homothety} follows.
\end{proof}

\begin{prop}\label{prop-char}
Let $(\wt{M},\mb{c})$ be a conformal spin structure of split signature $(n,n)$ satisfying the properties (a)--(d) of Theorem \ref{thm-char}. 
Then the local leaf space of the integrable distribution associated to the pure twistor spinor admits a projective structure $\mb{p}$ such that $(\wt{M},\bf{c})$ is the conformal extension associated to $\mb{p}$.
\end{prop}

\begin{proof}
From Proposition \ref{prop-twistor-spinor2parallel} we know that, locally, we can find metrics $g$ and $\hat{g}$ in $\mb{c}$ such that the twistor spinor $\chi$ is parallel with respect to the corresponding Levi-Civita connections $\wt D$ and $\widehat{\wt{D}}$, and $\hat{g} = \ee^{2 \phi} g$ for some smooth function $\phi$ on $\wt{M}$ which is constant on the leaves of $V$.
From Lemma \ref{CKf2homothety} we know that $k^a$ is a homothety satisfying \eqref{eq-CKf2homothety}.
Since $\chi$ is parallel with respect to $\wt{D}$, the Schouten tensor $\wt{\Rho}_{ab}$ is annihilated by $V$, and thus 
the integrability condition \eqref{eq-weylcond} is equivalent to the condition \eqref{eq-Weyl-cond-PW} on the Riemann tensor $\wt{R}_{abcd}$.
The same argument applies also to $\widehat{\wt{D}}$ and the corresponding Riemann tensor.
We can therefore apply Proposition \ref{prop-desc-PW} to each of the metrics $g$ and $\hat{g}$ with respective special torsion-free affine connections $D$ and $\widehat{D}$ on $M$. 

We shall show that $D$ and $\widehat{D}$ are projectively related. The Levi-Civita connections $\wt{D}$  and $\widehat{\wt{D}}$ are related by
\begin{align*}
\widehat{\wt{D}}_a \xi^b = \wt{D}_a \xi^b + \Upsilon_a \xi^b + \Upsilon_c \xi^c \delta_a^b - \xi_a \Upsilon^b,
\end{align*}
where $\Upsilon_a = \wt{D}_a \phi$. Since $\phi$ is constant along the leaves of $V$, the corresponding $1$-form $\Upsilon_a$ is strictly horizontal.
Hence we consider both $\phi$ and $\Upsilon_a$ as the pull-back of a smooth function $\phi$ and a $1$-form $\Upsilon_A$  on the leaf space $M$, respectively.
Therefore, for $\xi^a$ being a projectable vector field on $\wt{M}$ and $\xi^A$ denoting its projection to $M$, the two underlying affine connections differ by 
$$
\widehat{D}_A \xi^B = D_A \xi^B + \Upsilon_A \xi^B + \Upsilon_C \xi^C \delta_A^B.
$$ 
That is why $D$ and $\widehat{D}$ are projectively equivalent, cf. \eqref{eq-Dhat-D}.

Finally, from Proposition \ref{prop-weight} it follows that $\wt{M}$ is locally identified with $T^*M(2)$.
\end{proof}

Combining propositions \ref{prop-conf-extension} and \ref{prop-char} we immediately obtain our characterization Theorem \ref{thm-char}.

The conformal Patterson--Walker metric constructed above is also equipped with another distinguished spinor as explained below.

\begin{prop}\label{prop-other-dist}
The conformal extension $(\wt{M},\mb{c})$ admits a spinor field $\eta_A \in \wt{\mc{E}}_A[1]$, which, for any choice of Patterson--Walker metric, takes the form
\begin{align}\label{eq-def-eta}
\eta_A & = \frac{1}{2\sqrt{2}} k^b \check{\eta}_{bA} \, .
\end{align}
This spinor is pure off the zero-set of $k$ and satisfies
\begin{align}\label{eq-eta-etab}
\eta^a_{A'} \check{\eta}_{aB} = - 2 \, \eta_B \check{\eta}_{A'} \, ,
\end{align}
where $\eta^a_{A'} := \eta_B \gamma \ind{^a_{A'}^B}$, i.e.\ the totally isotropic $n$-plane distribution $U := \ker \check{\eta}_{aA'}$ intersects the horizontal distribution $H$ maximally and intersects the vertical distribution $V$ in the line distribution spanned by $k^a$. In particular, $k^a = 2 \sqrt{2} \eta_A \chi^{aA}$.

Further, $\eta_A$ satisfies the conformally invariant equation
\begin{align}\label{eq-other-spinor}
\wt{D}_a \eta_A - \frac{1}{\sqrt{2}} \gamma \ind{_a^{B'}_A} \check{\eta}_{B'} & = \frac{1}{8} k^d \wt{W}_{dabc} (\gamma ^b \gamma^c ) \ind{^B_A} \eta_B \, .
\end{align}
In particular, $\eta_A$ is a twistor spinor if and only if $(\wt{M},\mb{c})$ is conformally flat, i.e.\ $(M,\mb{p})$ is projectively flat.
\end{prop}

\begin{proof}
 That $\eta_A$ is pure follows from the fact that it lies in the image of $\check{\eta}_{aA}$ since $\check{\eta}_{A'}$ is a pure spinor. That it satisfies \eqref{eq-eta-etab} follows from a direct computation and commuting $\gamma$-matrices. Since $k^a$ and $\check{\eta}_{aA}$ have conformal weights $0$ and $1$ respectively, $\eta_A$  has conformal weight $1$ by \eqref{eq-def-eta}.

We now check that \eqref{eq-def-eta} is independent of the choice of connection in $\mb{p}$. Consider any two projectively related special connections $D_A$ and $\widehat{D}_A$ in $\mb{p}$ corresponding to horizontal distributions $H$ and $\widehat{H}$ on $\wt{M}$ annihilated by pure spinors $\check{\eta}_{A'}$ and $\widehat{\check{\eta}}_{A'}$ respectively. Note that with a choice of trivialization, \eqref{eq-Special-K} allows us to make the identification
\begin{align} \label{nup} 
\eta_A = \frac{1}{\sqrt{2}} p_A \, ,
\end{align}
and similarly for $\widehat{\check{\eta}}_{A'}$. Since the $1$-forms annihilating $H$ and $\widehat{H}$ are related as in \eqref{eq-Hhat-H} with $w=2$, we can then readily check that $\widehat{\check{\eta}}_{A'}$ and $\check{\eta}_{A'}$ are related by $\widehat{\check{\eta}}_{A'} = \check{\eta}_{A'} - \frac{1}{\sqrt{2}} \Upsilon_a \eta^a_{A'}$ where $\Upsilon_a = \Upsilon_A \chi_a^A$, or equivalently, by
\begin{align}\label{eq-hor-proj-ch}
\widehat{\check{\eta}}_{aA} & = \check{\eta}_{aA} + \sqrt{2} ( \eta_A \Upsilon_B - \eta_B \Upsilon_A ) \chi_a^B \, .
\end{align}
Since $k^a$ annihilates $\chi^{A'}$, the result follows immediately.

The final part of the proposition follows from a direct, albeit lengthy, computation.
\end{proof}

The identification \eqref{nup} will prove to be very convenient in explicit computations, and will be used ubiquitously in sections \ref{sec-einstein} and \ref{sec-infsym}.

\begin{rema} We can investigate the geometric properties of the distributions $V = \ker \chi_a^A$, $U =\ker \check{\eta}_{aA'}$ and $V \cap U = \langle k^a \rangle$ viewed as $\mathrm{G}$-structures on $\wt{M}$ with structure group taken to be the stabilizer of $\langle \chi^{A'} \rangle$, $\langle \eta_A \rangle$ or $\langle k^a \rangle$ in $\Spin(n,n)$ at a point. These can be expressed in terms of differential conditions on the fields $\chi^{A'}$, $\eta_A$ or $k^a$ defined up to scale, and are related to the notion of intrinsic torsion of the $\mathrm{G}$-structure. For pure spinor fields, this is the topic of the articles \cite{Taghavi-Chabert2012a,Taghavi-Chabert2013}, to which we refer for details.
\begin{enumerate}[(a)]
\item For $\chi^{A'}$ parallel, the intrinsic torsion is trivial. This implies in particular $\left( \chi^{aA} \wt{D}_a \chi^{bB} \right)  \chi_b^C = 0$, i.e.\ $V$, as any integrable totally isotropic $n$-plane distribution on $(M,\mb{c})$, is totally geodetic \cite{hughston-mason,taghavi-chabert-goldberg,Taghavi-Chabert2012a}.
\item From \eqref{eq-other-spinor}, we deduce $\left( \eta^a_{A'} \wt{D}_a \eta^b_{B'} \right) \eta_{bC'} = k^d \wt{W}_{dabc} \eta^a_{A'} \eta^b_{B'} \eta^c_{C'}$, 
which by the Bianchi identity implies that $\left( \eta^a_{[A'} \wt{D}_a \eta^b_{B'} \right) \eta_{bC']} = 0$.
The distribution $U$ is integrable, i.e.\ $\left( \eta^a_{A'} \wt{D}_a \eta^b_{B'} \right) \eta_{bC'} =0$, if and only if $(\wt{M},\mb{c})$ is conformally flat.
\item Being light-like and conformal Killing, $k^a$ generates a \emph{shear-free congruence of null geodesics} tangent to $U \cap V$, i.e.\ $\left( k^c \wt{D}_c k^{[a} \right) k^{b]} = 0$ and $\mc{L}_k g_{ab} = f \, g_{ab} +  t_{(a} k_{b)}$ for some function $f$ and $1$-form $t_a$.
\item Moreover, this congruence is also \emph{twisting}, i.e.\ $k_{[a} \wt{D}_b k_{c]}$ does not vanish. Since $k_a$ annihilates the rank-$(2n-1)$ distribution $U + V$, this means that $U + V$ is not integrable.
\end{enumerate}
\end{rema}

\begin{rema}
In four dimensions, i.e.\ $n=2$, we can identify $T \wt{M}$ with $\wt{S}_+ \otimes \wt{S}_-$, and use the two-spinor calculus of \cite{penrose-rindler-84}. We can choose a spin invariant skew-symmetric bilinear form $\varepsilon_{AB}$ on $\widetilde{S}_-$, with inverse $\varepsilon^{AB}$, to be preserved by the Levi-Civita connection $\wt{D}_a$ of a Patterson--Walker metric in $\mb{c}$, and identify $\varepsilon_{AB}$ as the volume form on $M$ preserved by the corresponding special connection $D_A \in \mb{p}$. It can be shown \cite{bryant-recent,dunajski-2002} that the function $\Theta_{AB}$ can be expressed in terms of a single function $\Theta = \Theta (x,p)$, i.e.\ 
 $\Theta_{AB}  =  \varepsilon_{AC}\varepsilon_{BD} \parderv{^2}{p_C \partial p_D} \Theta$.
 Then equations \eqref{eq-homogeneous_Theta} and \eqref{eq-linear_Theta} tell us that the function $\Theta$ must be a polynomial of degree $3$ in the coordinates $p_A$, i.e.\ where $\Gamma \ind{_A^B_C}$ are the Christoffel symbols for an affine connection on the projective surface $M$.

The Weyl tensor can be expressed as
\begin{align*}
\wt{W}_{abcd} & = \wt{\Psi}_{A'B'C'D'} \varepsilon_{AB} \varepsilon_{CD} + \wt{\Psi}_{ABCD} \varepsilon_{A'B'} \varepsilon_{C'D'} \, ,
\end{align*}
where $\wt{\Psi}_{A'B'C'D'}$ and $\wt{\Psi}_{ABCD}$ are the self-dual and anti-self-dual parts of the Weyl tensor.   Writing $v^a  = \chi^{A'} v^A$ and $w^a  = \chi^{A'} w^A$ for two arbitrary elements of $V$ for some spinors $v^A$ and $w^A$, we see that \eqref{eq-weylcond} is equivalent to
\begin{align*}
\chi^{A'} \chi^{C'} \wt{\Psi}_{A'B'C'D'} v_B w_D + v^A w^C \wt{\Psi}_{ABCD} \chi_{B'} \chi_{D'} & = 0 \, .
\end{align*}
The integrability condition $\widetilde{W}_{abcd} \gamma^c \gamma^d \chi=0$ for the existence of a twistor spinor $\chi^{A'}$ is $\wt{\Psi}_{A'B'C'D'} \chi^{A'} = 0$, i.e.\ the self-dual Weyl tensor is of Petrov type N. Combined with \eqref{eq-weylcond}, we immediately conclude $\wt{\Psi}_{ABCD} = 0$, i.e.\ the Weyl tensor is self-dual.
\end{rema}

\section{Einstein metrics}\label{sec-einstein}
We say that a non-trivial density $\wt{\sigma} \in \wt{\mc{E}}[1]$ is an  \emph{almost Einstein scale} if the metric $g_{ab} = \wt{\sigma}^{-2} \mb{g}_{ab}$ is Einstein off the zero-set of $\wt{\si}$, i.e.\ $\wt{\Ric}_{ab} = \lambda \, g_{ab}$ for some constant $\lambda$. One can show that this is equivalent to $\wt{\sigma}$ satisfying the conformally invariant equation
\begin{align}\label{eq-aes}
\bigl( \wt{D}_{(a} \wt{D}_{b)} + \wt{\Rho}_{ab} \bigr)_0 \wt{\sigma} & = 0 \, .
\end{align}
We now show that any Einstein scale on $(\wt{M},\mb{c})$ gives rise to solutions to overdetermined projectively invariant differential equations.

One of these is a projective analogue of equation \eqref{eq-aes}, to be precise, a solution $\sigma \in \mc{E}(1)$ to
\begin{align}\label{eq-ars}
\bigl( D_{(A} D_{B)} + \Rho_{AB} \bigr) \sigma & = 0 \, .
\end{align}
Away from their singularity sets,  solutions to this equation determine Ricci-flat affine connections $D_A$ in $\mb{p}$. Thus, they are sometimes referred to as \emph{almost Ricci-flat scales}.

We shall also consider a generalization of Euler vector fields to weighted vector fields: i.e.\ a solution $\xi^A \in \mc{E}^A(-1)$ satisfying
\begin{align}\label{eq-s+2s}
& D_A \xi^B - \frac{1}{n} \delta_A^B (D_C \xi^C)  = 0 \, ,
\end{align}
Equation \eqref{eq-s+2s} implies
\begin{align}
& D_{(A} D_{B)} \xi^C + \delta_{(A}^C P_{B)D} \xi^D  = 0 \, ,  \label{eq-s+2s-prol} \\
& W \ind{_{AB}^C_D} \xi^D   = 0 \, . \label{eq-s+2s-int-cond}
\end{align}

With reference to Lemma \ref{lemarkable} and the fact that $\eta_A$ has conformal weight 1  we prove:
\begin{lem}\label{lem-E-scale-WP}
Let $\sigma \in \mc{E}(1)$ and $\xi^A \in \mc{E}^A(-1)$. 
Then 
\begin{align}\label{eq-local-Einstein-scales}
\wt{\sigma}_-  & :=  \pi^*\sigma  \, , &
\wt{\sigma}_+ & := \sqrt{2} \, \xi^A \eta_A  \, 
\end{align}
are sections of $\wt{\mc{E}}[1]$.
Here, $\pi$ is the projection from $\wt{M}$ to $M$, and $\xi^A$ is viewed as a section of $\wt{\mc{E}}^A$.
\end{lem}

Before we proceed, we note that for any $k^a \in \wt{\mc{E}}^a$, $\sigma \in \wt{\mc{E}}[w]$ and $\sigma^a \in \wt{\mc{E}}^a[w]$, we have
\begin{align*}
\mc{L}_k \wt{\sigma} & = k^a \wt{D}_a \wt{\sigma} - \frac{w}{2n} \wt{\sigma} \wt{D}_a k^a \, , &
\mc{L}_k \wt{\sigma}^a & = k^b \wt{D}_b \wt{\sigma}^a - \wt{\sigma}^b \wt{D}_b k^a - \frac{w}{2n} \wt{\sigma}^a \wt{D}_b k^b \, .
\end{align*}
Choosing a Patterson--Walker metric, these simplify to
\begin{align}\label{eq-Lie-density}
\mc{L}_k \wt{\sigma} & = k^a \wt{D}_a \wt{\sigma} - w \wt{\sigma} \, , &
\mc{L}_k \wt{\sigma}^a & = k^b \wt{D}_b \wt{\sigma}^a - \wt{\sigma}^b \mu \ind{_b^a} - (w + 1) \wt{\sigma}^a \, , 
\end{align}
where we have made use of \eqref{eq-CKf2homothety}. Similar formulae for the Lie derivative on weighted forms can be obtained using the Leibniz rule or the fact that $k^a$ is a conformal Killing field.

\begin{lem}
  The lifts satisfy $\mc{L}_k \wt{\sigma}_{\pm} = \pm \wt{\sigma}_{\pm}$. 
\end{lem}
\begin{proof}
By \eqref{eq-Special-K}, we have $k^a \wt{D}_a \wt{\sigma}_+ = 2 \, \wt{\sigma}_+$ and $k^a \wt{D}_a \wt{\sigma}_- = 0$. Applying \eqref{eq-Lie-density} with $w=1$ completes the proof.
\end{proof}

\begin{prop}\label{prop-Einsteinlifts}
\hfill
\begin{enumerate}[(a)]
\item Suppose $\sigma \in \mc{E}(1)$ satisfies \eqref{eq-ars}. Then its lift $\wt{\sigma}_-$ given by \eqref{eq-local-Einstein-scales} is an almost Einstein scale, i.e.\ a solution to \eqref{eq-aes}.
\item Suppose $\xi^A \in \mc{E}^A(-1)$ satisfies \eqref{eq-s+2s} together with the integrability condition
\begin{align}\label{eq-int-cond-E(-1)}
\xi^D W \ind{_{DA}^C_B} & = 0 \, .
 \end{align} Then its lift $\wt{\sigma}_+$ given by \eqref{eq-local-Einstein-scales} is an almost Einstein scale.
\end{enumerate}
In both cases, the rescaled metrics they define are Ricci-flat off the singular sets of $\wt{\sigma}_\pm$.
\end{prop}

\begin{proof}
\begin{enumerate}[(a)]
\item Let $\sigma$ be a Ricci-flat scale with associated torsion-free affine connection $D_A$ in $\mb{p}$ on $M$ that is Ricci-flat, i.e.\ $\Rho_{AB}=0$. Then $D_A$ is special and determines a Patterson--Walker metric $g$ with corresponding conformal scale $\wt{\sigma}_-$ as given by \eqref{eq-local-Einstein-scales}. Reading off \eqref{eq-Walker-Rho}, we see that $\wt{\Rho}_{ab} = 0$, i.e.\ $g$ is Ricci-flat.
\item Let us rewrite \eqref{eq-local-Einstein-scales} as $\wt{\sigma}_+ = \frac{1}{2} \left( \xi^A \check{\eta}^a_A \right) k_a$. Then, using the Leibniz rule, \eqref{eq-Dv-no-p}, with $v^A = \xi^A$ and $\alpha_A=0$, \eqref{eq-mu-eigen} and \eqref{eq-CKf2homothety}, we obtain
\begin{align}\label{eq-Ds+}
\wt{D}_a \wt{\sigma}_+ & =  \left( D_A \xi^B \right) p_B \chi_a^A + \xi^B \check{\eta}_{aB} \, .
\end{align}
Similarly,
\begin{align}\label{eq-DDs+}
\wt{D}_a \wt{D}_b \wt{\sigma}_+ & = \left( D_A D_B  \xi^C  - \xi^D R \ind{_{D B}^C_A} \right) \chi_a^A \chi_b^B p_C + 2 \left( D_A \xi^B \right) \chi_{(a}^A \check{\eta}_{b)B} \, .
\end{align}
Finally, using \eqref{eq-Riem2AffCurv}, \eqref{formulaCprojective} and \eqref{eq-Walker-Rho}, we find
\begin{multline}\label{eq-key}
 \left( \wt{D}_{(a} \wt{D}_{b)} \wt{\sigma}_+ + \wt{\Rho}_{ab} \wt{\sigma}_+ \right)_0 =  2 \left( D_A \xi^B - \frac{1}{n} \delta_A^B D_C \xi^C \right) \chi_{(a}^A \check{\eta}_{b)B} \\
 + \left( D_A D_B \xi^C + \delta_A^C P_{BD} \xi^D - \xi^D  W \ind{_{DA}^C_B} \right) p_C \chi_a^A \chi_b^B \, .
\end{multline}
That $\wt{\sigma}_+$ is an almost Einstein scale follows immediately from \eqref{eq-s+2s}, \eqref{eq-s+2s-prol} and \eqref{eq-int-cond-E(-1)}. To show that the rescaled metric is Ricci-flat, we compute the trace $\widehat{\wt{\Rho}}$ of the Rho tensor of the rescaled metric via the transformation rule $\widehat{\wt{\Rho}} = \wt{\Rho} - \wt{D}^a \Upsilon_a + (1-n) \Upsilon^a \Upsilon_a$, where $\wt{\Rho} := \wt{\Rho} \ind{^a_a}$ and $\Upsilon_a := - \sigma_+^{-1} \wt{D}_a \sigma_+$. Using \eqref{eq-Ds+}, \eqref{eq-DDs+} and the fact $\wt{\Rho}=0$ for a Patterson--Walker metric, one easily verifies $\widehat{\wt{\Rho}} = 0$ as required.
\end{enumerate}
\end{proof}

\begin{lem}
Let $\wt{\sigma} \in \wt{\mc{E}}[1]$ with $\mc{L}_k \wt{\sigma} = r \, \wt{\sigma}$ for some real constant $r$. Then $\wt{\sigma}$ is homogeneous of degree $\frac{r+1}{2}$ in $p_A$. In particular, $\wt{\sigma}_+$ is homogeneous of degree $1$ and $\wt{\sigma}_-$ of degree $0$.
\end{lem}

\begin{proof}
This follows from \eqref{eq-Lie-density} with $w=1$ and \eqref{eq-Special-K}.
\end{proof}

\begin{prop}\label{prop-Einstein2lifts}
Let $\wt{\sigma} \in \wt{\mc{E}}[1]$ be an almost Einstein scale. Then 
\begin{align*}
\wt{\sigma} & = \wt{\sigma}_+ + \wt{\sigma}_-
\end{align*}
where $\mc{L}_k \wt{\sigma}_\pm = \pm \wt{\sigma}_\pm$. Further, for any choice of Patterson--Walker metric, $\wt{\sigma}_\pm$ can be expressed as the lifts \eqref{eq-local-Einstein-scales}, where
\begin{enumerate}[(a)]
\item $\sigma = \wt{\sigma}_-(x)$ is an almost Ricci-flat scale on $(M,\mb{p})$.
\item $\xi^A = \chi^{aA} \wt{D}_a \wt{\sigma}_+$ satisfies \eqref{eq-s+2s} together with the integrability condition
\eqref{eq-int-cond-E(-1)}.
\end{enumerate}
\end{prop}

\begin{proof}
We use a Patterson--Walker metric throughout. Using \eqref{eq-Lie-density} with $w=1$, together with the Leibniz rule and the fact that $ \mu \ind{^a_b} k^b = -k^a$, we compute
\begin{align*}
\mc{L}_k^2 \wt{\sigma} = k^a k^b \wt{D}_a \wt{D}_b \wt{\sigma} + \wt{\sigma} \, .
\end{align*}
Since $\wt{\sigma}$ is an almost Einstein scale,  $k^a k^b \left( \wt{D}_{(a} \wt{D}_{b)} + \wt{\Rho}_{ab} \right)_0 \wt{\sigma} = k^a k^b \wt{D}_a \wt{D}_b \wt{\sigma} = 0$, where we have used the fact that, for a Patterson--Walker metric, $\wt{\Rho}_{ab}k^b =0$ by \eqref{eq-Walker-Rho}.  Hence $\mc{L}_k^2 \wt{\sigma} = \wt{\sigma}$, i.e.\ $(\mc{L}_k - 1)(\mc{L}_k + 1) \wt{\sigma}=0$. This equation is the characteristic polynomial for $\mc{L}_k$ viewed as a linear operator acting on the finite-dimensional space of Einstein scales, and the decomposition of this space follows immediately. Details and generalizations are given in \cite{gover-silhan-commuting}.

Next, assume that $\wt{\sigma}_\pm$ are almost Einstein scales with $\mc{L}_k \wt{\sigma}_\pm = \pm \wt{\sigma}_\pm$, so that $\chi^{aA} \chi^{bB} \wt{D}_a \wt{D}_b \wt{\sigma}_\pm = 0$. In coordinates, this condition reads
\begin{align*}
\parderv{{}^2}{p_A p_B} \wt{\sigma}_\pm & = 0 \, .
\end{align*}
This means that $\wt{\sigma}_\pm$ are polynomials of degree 1 in $p_A$ with coefficients depending on $x^A$ only, i.e.\ $\wt{\sigma}_\pm = \xi^A p_A + \sigma$, where $\xi^A = \xi^A(x)$ and $\sigma=\sigma(x)$. Now, using \eqref{eq-Lie-density} with $w=1$, $\mc{L}_k \wt{\sigma}_\pm = \pm \wt{\sigma}$ can be recast as
\begin{align*}
k^a \wt{D}_a \wt{\sigma}_+ & = 2 \, \wt{\sigma}_+ \, , & k^a \wt{D}_a \wt{\sigma}_- & = 0 \, .
\end{align*}
Using \eqref{eq-Special-K}, these conditions tell us that $\wt{\sigma}_+$ is homogeneous of degree 1 in $p_A$ and $\wt{\sigma}_-$ homogeneous of degree 0 in $p_A$. Since they are also polynomials in $p_A$, we conclude that $\wt{\sigma}_\pm$ take the form \eqref{eq-local-Einstein-scales}.

For the last part of the proposition, we assume $\wt{\sigma}_\pm$ are almost Einstein scales with $\mc{L}_k \wt{\sigma}_\pm = \pm \wt{\sigma}_\pm$ so that $\wt{\sigma}_\pm$ are given by \eqref{eq-local-Einstein-scales}. We proceed as follows.
\begin{enumerate}[(a)]
\item The  almost Einstein scale $\wt{\si}_-$ defines a conformally related Patterson--Walker metric $\wt{\sigma}_-^{-2} g_{ab}$ with $\wt{\Rho}_{ab} = 0$. By \eqref{eq-Walker-Rho}, we conclude immediately $\Rho_{AB}=0$, i.e.\ the corresponding affine connection on $M$ is Ricci-flat.
\item
Equation \eqref{eq-aes} with $\wt{\sigma}=\wt{\sigma}_+$ implies that the left-hand side of \eqref{eq-key} vanishes, and in particular, each term of the right-hand side must vanish separately, i.e.\ $\xi^A$ satisfies \eqref{eq-s+2s} and
\begin{align}
D_{(A} D_{B)} \xi^C + \delta_{(A}^C \Rho_{B)D} \xi^D - \xi^D  W \ind{_{D(A}^C_{B)}} & = 0 \, . \label{eq-2-DDs}
\end{align}
But with reference to \eqref{eq-s+2s-int-cond} and \eqref{eq-s+2s-prol}, together with the Bianchi identity, equation \eqref{eq-2-DDs} implies \eqref{eq-int-cond-E(-1)}, i.e.\ $\xi^D W \ind{_{DA}^C_B} = 0$.
\end{enumerate}
\end{proof}

Combining Proposition \ref{prop-Einsteinlifts} and Proposition \ref{prop-Einstein2lifts} now gives Theorem \ref{thm-scales}.

\section{Symmetries}\label{sec-infsym}
We now show that any conformal Killing vector $\wt{v}^a$ on $(\wt{M},\mb{c})$, i.e.\ a solution of
\begin{align}
\wt{D}_a \wt{v}_b & = \wt{\phi}_{ab} - \wt{\psi} \, \mb{g}_{ab} \, , \label{eq-conf-iso-prol1}
\end{align}
where $\wt{\phi}_{ab} = \wt{D}_{[a} \wt{v}_{b]}$ and $\wt{\psi} = - \frac{1}{2n} \mb{g}^{ab} \wt{D}_a \wt{v}_b$,  gives rise to solutions of overdetermined projectively invariant differential equations on $(M, \mb{p})$ as claimed by Theorem \ref{thm-confkillingfields}.

Before we proceed, we recall the prolongation equations for \eqref{eq-conf-iso-prol1}:
\begin{align}
\wt{D}_a \wt{\psi} & = \wt{\Rho}_{ab} \wt{v}^b - \wt{\beta}_a \, , \label{eq-conf-iso-prol3} \\
\wt{D}_a \wt{\phi}_{bc} & = -2 \, \mb{g}_{a[b} \wt{\beta}_{c]} - 2 \, \wt{\Rho}_{a[b} \wt{v}_{c]} + \wt{v}^d \wt{W}_{dabc} \, , \label{eq-conf-iso-prol2} \\
\wt{D}_a \wt{\beta}_b & = \wt{\Rho} \ind{_a^c} \wt{\phi}_{cb} - \wt{\psi} \, \wt{\Rho}_{ab} - \wt{v}^d Y_{abd} \, . \label{eq-conf-iso-prol4}
\end{align}
Here, $\wt{\beta}_a$ is defined by \eqref{eq-conf-iso-prol3}.

\subsection{Projectively invariant equations}\label{sec-pro-inv-eq}
An \emph{infinitesimal projective symmetry} is a vector field $v^A$ that preserves the projective structure, i.e.\ for any $D_A$ in $\mb{p}$ and vector field $X^A$,
\begin{align}\label{eq-proj-symG}
\mc{L}_v D_A X^B - D_A \mc{L}_v X^B
& = Q \ind{_{AC}^B} X^C \, , & \mbox{where} & & Q \ind{_{AB}^C} & = 2 \, \delta_{(A}^C \Upsilon_{B)} \, ,
\end{align}
for some $1$-form $\Upsilon_A$.

It can be shown that $v^A$ is an infinitesimal projective symmetry if and only if it satisfies the following projectively invariant overdetermined system of partial differential equations \cite{eastwood-notes}
\begin{align}\label{eq-proj-sym}
\left( D_{(A} D_{B)} v^C + \Rho_{AB} v^C + v^D W \ind{_{D(A}^C_{B)}} \right)_0 & = 0 \, .
\end{align}
Define
\begin{align}\label{eq-proj-sym-xtra}
\begin{aligned}
&\phi_A^B  := D_A v^B - \frac{1}{n} \delta_A^B D_C v^C \, \quad  \psi  := \frac{1}{n} D_C v^C, \\
& \beta_A := - \frac{1}{n+1} D_A D_B v^B -  \Rho_{AB} v^B \, .
\end{aligned}
\end{align}
 Then, under a projective transformation, using \eqref{eq-Dhat-D}, the fields transform as
 \begin{align}
 \begin{aligned}\label{eq-transf-proj-sym}
 \hat{v}^A & = v^A \, , & \hat{\phi}_B^A & = \phi_B^A - \frac{1}{n} \Upsilon_C v^C \delta_B^A + \Upsilon_B v^A \, , \\ \hat{\psi} & = \psi + \frac{n+1}{n} \Upsilon_C v^C \, , &
 \hat{\beta}_A & = \beta_A -  \Upsilon_B \phi^B_A - \Upsilon_A \psi - \Upsilon_A \Upsilon_B v^B \, .
 \end{aligned}
 \end{align}
Equation \eqref{eq-proj-sym} can be written in prolonged form as
\begin{align}
&D_A v^B - \phi_A^B - \delta_A^B \psi  = 0 \, ,\notag\\
&D_A \psi  +  \frac{n+1}{n} \left( \beta_A + \Rho_{AB} v^B \right)  = 0 \, ,\notag \\
&D_{(A} \phi_{B)}^C + \Rho_{AB} v^C \!+ v^D W \ind{_{D(A}^C_{B)}} \!-  \frac{1}{n} \delta_{(A}^C \left(  \Rho_{B)D} v^D \!- (n\!-\!1) \beta_{B)} \right)  = \!0 \, , \label{eq-proj-sym-prol3}\\\notag
&D_A \beta_B - \Rho_{AB} \psi - \Rho_{AC} \phi_B^C -v^C Y_{ABC}  = 0 \, . 
\end{align}
The first two equations immediately follow from \eqref{eq-proj-sym-xtra}, the third one from \eqref{eq-proj-sym}, and the last one from the divergence of the latter equation.

Next, we shall consider the following projectively invariant equation
\begin{align}
    D_C w^{AB} + \frac{2}{n-1} \, \delta_C^{[A} D_D w^{B]D} & = 0\, , \label{eq-bivec}
  \end{align}
where  $w^{AB} \in \mc{E}^{[AB]}(-2)$. Defining 
\begin{align}\label{eq-bivec-xtra}
&\nu^A  := \frac{1}{n-1} D_C w^{CA} \, ,
\end{align}
one can easily verify the transformation rules under a projective change
 \begin{align}\label{eq-transf-bivec}
 \hat{w}^{AB} & = w^{AB} \, , & \hat{\nu}^A & = \nu^A - w^{AB} \Upsilon_B \, .
 \end{align}
Differentiating \eqref{eq-bivec-xtra}, one can show that equation \eqref{eq-bivec} is equivalent to the system
\begin{align}\label{eq-bivec-prol}
\begin{aligned}
&D_C w^{AB} - 2 \, \delta_C^{[A} \nu^{B]}  = 0\, , \\ 
   &D_A \nu^B + \Rho_{AC} w^{CB} + \frac{1}{2(n-2)} w^{CD} W \ind{_{CD}^B_A} = 0 \, . 
\end{aligned}
  \end{align}

Finally, we shall consider a weighted $1$-form $\alpha_A \in \mc{E}_A(2)$ that satisfies the \emph{Killing equation}
\begin{align}\label{eq-proj-Killing}
D_{(A} \alpha_{B)} & = 0 \, .
\end{align}

\subsection{Projectively invariant lifts}
Let $v^A \in \mc{E}^A$, $w^{AB} \in \mc{E}^{[AB]}(-2)$ and $\alpha_A \in \mc{E}_A(2)$, and  $\phi_A^B$, $\psi$,  and $\nu^A$ are given by \eqref{eq-proj-sym-xtra} and \eqref{eq-bivec-xtra}. At this stage, we do not assume that $v^A$, $w^{AB}$ and $\alpha_A$ satisfy \eqref{eq-proj-sym}, \eqref{eq-bivec} and \eqref{eq-proj-Killing} respectively.

\begin{lem}\label{lem-lifts}
Choosing a special torsion-free affine connection $D\in\mb{p}$, we define the vector fields 
 \begin{align}
 \wt{v}_0^a & :=v^A \check{\eta}_A^a  - \sqrt{2} \phi_B^A \eta_A \chi^{aB} + \frac{n-1}{2(n+1)} \psi k^a \,  , \label{eq-lift0}  
\\
   \wt{v}_+^a & := \sqrt{2} w^{AB} \eta_A \check{\eta}^a_B - \frac{1}{\sqrt{2}} (\nu^B \eta_B )  k^a \, , \label{eq-lift+}
 \\
       \wt{v}_-^a & := \alpha_A \chi^{aA} \, , \label{eq-lift-} 
\end{align}
on $\wt{M}$. Then the forms of these vectors are independent of the choice of $D \in \mb{p}$.
\end{lem}

\begin{proof}
We first check the conformal weight of each expression using Lemma \ref{lemarkable}. For instance, we view $w^{AB}$ and $\nu^A$ as sections of $\wt{\mc{E}}^{AB}$ and $\wt{\mc{E}}^A[1]$ respectively. Since $\eta_A$ and $\check{\eta}^a_B$ have weight $1$ and $-1$ respectively, we see that the both terms in \eqref{eq-lift+}, and thus $\wt{v}^a_+$, have weight $0$ as required.

Next, under a projective change of affine connections in $\mb{p}$, $\eta_A$, $\chi_a^A$ and $k^a$ are invariant. In particular,  $\wt{v}_-^a$ is invariant. The invariance of $\wt{v}_0^a$ and $\wt{v}_+^a$ can be verified by observing that the change of horizontal distribution as given \eqref{eq-hor-proj-ch} induced by a projective change, and using \eqref{nup}, counterbalances the transformation rules \eqref{eq-transf-proj-sym} and \eqref{eq-transf-bivec}.
\end{proof}

\begin{lem}\label{lem-lift-properties}
The vector fields in Lemma \ref{lem-lifts} satisfy the following properties:
\begin{enumerate}[(a)]
\item  $\mc{L}_k \wt{v}_\pm^a = \pm 2 \, \wt{v}^a_\pm$ and $\mc{L}_k
\wt{v}^a_0 = 0$;\label{lem-a}

\item $\wt{v}^a_+ $ and $\wt{v}^a_-$ are tangent to $U= \ker \eta_{aA'}$ and $V= \ker \chi_a^A$ respectively, i.e.\  $\wt{v}^a_+ \eta_{aA'} = 0$ and  $\wt{v}^a_- \chi_a^A = 0$;

\item $\wt{v}_0^a$ is not tangent to $U + V = \ker k_a$, i.e.\ $\wt{v}^a_0 k_a$ is not identically zero.
\end{enumerate}
\end{lem}

\begin{proof}
\begin{enumerate}[(a)]
\item First observe that $[2p_A \parderv{}{p_A}, \parderv{}{p_B}] = -2
\parderv{}{p_B}$
which, using  \eqref{eq-for-Josef} is equivalent to the first relation
in the display
\begin{align}\label{eq-com-k}
[ k^a \wt{D}_a , \chi^{bA} \wt{D}_b] & = - 2 \, \chi^{bA} \wt{D}_b \, ,
&  [k^a \wt{D}_a ,
\check{\eta}^b_A \wt{D}_b] & = 0 \, , & \mc{L}_k p_A & = 2 \, p_A.
\end{align}
The second relation follows similarly using \eqref{eq-for-Josef} and the
last one is obvious.
Further, $\mc{L}_k v^A=0$ and similarly for all sections depending only on $x^A$.
Using \eqref{nup}, these relations and the Leibniz rule, it is a straightforward computation to verify 
part (\ref{lem-a}).

\item  Here $\wt{v}^a_- \chi_a^A = 0$ follows from \eqref{xinu}. Further recall 
$\eta^a_{A'} \check{\eta}_{aB} = - 2 \, \eta_B \check{\eta}_{A'}$ from
\eqref{eq-eta-etab}. Since
$w^{AB}$ is skew-symmetric, the first summand of $\wt{v}^a_+$ inserts trivially
into $\eta_{aA'}$. The second summand inserts trivially using
\eqref{eq-def-eta} since $k$ is null.

\item It follows from \eqref{nup}, and the properties of $k^a$ and $\chi_a^A$ that
$\wt{v}^a_0 k_a =2 \, v^A p_A$
which is zero if and only if $v^A$ is zero. But if $v^A$ is zero then $\phi_A^B$ and $\psi$ are zero since they are defined by \eqref{eq-proj-sym-xtra}.
\end{enumerate}
\end{proof}

\begin{prop}\label{prop-confkillinglift}
\begin{enumerate}[(a)]\hfill
\item Suppose $v^A \in \mc{E}^A$ is an infinitesimal projective symmetry,  i.e.\ satisfies \eqref{eq-proj-sym}. Then its lift $\wt{v}_0^a$ given by \eqref{eq-lift0} is a conformal Killing field.
\item Suppose $w^{AB} \in \mc{E}^{[AB]}(-2)$ satisfies \eqref{eq-bivec}  together with the integrability condition
\begin{align}\label{eq-int-cond-bivec}
    w^{B(A} W \ind{_{B(C}^{D)}_{E)}}=0 \, .
  \end{align}
Then its lift $\wt{v}_+^a$ given by \eqref{eq-lift+} is a conformal Killing field.
\item Suppose  $\alpha_A \in \mc{E}_A(2)$ satisfies the Killing equation \eqref{eq-proj-Killing}. Then its lift $\wt{v}_-^a$ given by \eqref{eq-lift-} is a (conformal) Killing field.
\end{enumerate}

\end{prop}

\begin{proof}
In the following we work with a choice of Patterson--Walker metric $g\in\mb{c}$.
\begin{enumerate}[(a)]
\item Suppose $v^A$ satisfies \eqref{eq-proj-sym}, so that $\phi_A^B$, $\psi$ and $\beta_A$ are given by \eqref{eq-proj-sym-xtra}, and lift $v^A$ to $\wt{v}^a:=\wt{v}^a_0$ as given by \eqref{eq-lift0}. Then, using \eqref{eq-Dv-no-p},
\begin{multline}\label{eq-projsym2CKf}
\wt{D}_{(a} \wt{v}_{b)} = 
\left( D_A v^{B} - \frac{1}{n} D_C v^C \delta_A^B  - \phi_A^B \right) \chi_{(a}^A \check{\eta}_{b)B} 
+ \frac{1}{2} \left( \frac{1}{n} D_C v^C + \frac{n-1}{n+1} \,  \psi \right) g_{ab} \\ 
- \Bigl( D_A \phi_B^C + \Rho_{AB} v^C + v^D W \ind{_{DA}^C_B}  \\
 -  \frac{1}{n} \delta_A^C \left( \Rho_{BD} v^D - (n-1) \beta_B \right) \Bigr) p_C \chi_{(a}^A \chi_{b)}^B  \, .
\end{multline}
Since \eqref{eq-proj-sym} is equivalent to \eqref{eq-proj-sym-prol3}, it is clear that the first and third terms of \eqref{eq-projsym2CKf} vanish, and so \eqref{eq-projsym2CKf} is proportional to the metric, i.e $\wt{v}^a$ is a conformal Killing field.

\item  Suppose $w^{AB}$ satisfies \eqref{eq-bivec}, and lift $w^{AB}$ to $\wt{v}^a:=\wt{v}^a_+$ as given by \eqref{eq-lift+}.  Then, using \eqref{eq-Dv-no-p},
\begin{multline}\label{eq-bivec2CKf}
\wt{D}_{(a} \wt{v}_{b)} = \left( D_A w^{BC} - 2 \delta_A^{[B} \nu^{C]}  \right) \chi_{(a}^A \check{\eta}_{b)B} p_C  -  (\nu^C p_C) g_{ab} \\ 
- \chi_{(a}^A \chi_{b)}^B \left( \left( D_A \nu^C + \Rho_{AE} w^{EC} \right) \delta_B^D - w^{EC} W \ind{_{EA}^D_B} \right) p_C p_D \, .
\end{multline}
Since  \eqref{eq-bivec} is equivalent to \eqref{eq-bivec-prol}, and we assume in addition \eqref{eq-int-cond-bivec}, we immediately conclude $\wt{D}_{(a} \wt{v}_{b)} =  -  (\nu^C p_C) g_{ab}$, i.e.\ $\wt{v}^a$ is conformal Killing.
\item
 Suppose $\alpha_A$ is a solution to \eqref{eq-proj-Killing}, and lift $\alpha_A$ to $\wt{v}^a:=\wt{v}^a_-$ as given by \eqref{eq-lift-}. Then, using \eqref{eq-Dv-no-p},
\begin{align}\label{eq-projKilling2CKf}
\wt{D}_{(a} \wt{v}_{b)} & = \left( D_A \alpha_B \right) \chi_{(a}^A \chi_{b)}^B \, .
\end{align}
By \eqref{eq-proj-Killing}, we now conclude  that $\wt{v}^a$ is a (conformal) Killing field.

\end{enumerate}
\end{proof}

\subsection{Decomposition of conformal Killing fields}

Before we proceed, we record the following technical lemma.
\begin{lem}\label{lem-hom}
Let $\wt{v}^a \in \wt{\mc{E}}^a$ be a vector field on $\wt{M}$. Choose a Patterson--Walker metric so that $\wt{v}^a = \wt{v}^A \check{\eta}^a_A + \wt{\alpha}_A \chi^{aA}$ for some $\wt{v}^A$ and $\wt{\alpha}_A$. Then $\mc{L}_k \wt{v}^a = 2 \, r \, \wt{v}^a$ for some real constant $r$ if and only if $\wt{v}^A$ and $\wt{\alpha}_A$ are homogeneous of degree $r$ and $r +1$ in $p_A$ respectively. In particular,
\begin{enumerate}[(a)]
\item  $\mc{L}_k \wt{v}^a = 0$ if and only if $\wt{v}^A$ and $\wt{\alpha}_A$ are homogeneous of degree $0$ and $1$ in $p_A$ respectively.

\item  $\mc{L}_k \wt{v}^a = 2 \, \wt{v}^a$ if and only if $\wt{v}^A$ and $\wt{\alpha}_A$ are homogeneous of degree $1$ and $2$ in $p_A$ respectively;

\item  $\mc{L}_k \wt{v}^a = -2 \, \wt{v}^a$ if and only if $\wt{v}^A$ and $\wt{\alpha}_A$ are homogeneous of degree $-1$ and $0$ in $p_A$ respectively;

\end{enumerate}
\end{lem}

\begin{proof}
This follows from \eqref{eq-Lie-density}, or \eqref{eq-com-k}, and \eqref{eq-Special-K}.
\end{proof}

\begin{prop}\label{prop-conf-iso2lift}
A conformal Killing field $\wt{v}^a \in \wt{\mc{E}}^a$  can be uniquely decomposed as
\begin{align}\label{eq-conf-iso2lift}
\wt{v}^a & = \wt{v}^a_+ +  \wt{v}^a _0 + \wt{v}^a _- + c \, k^a
\end{align}
where $\mc{L}_k \wt{v}_\pm^a = \pm 2 \, \wt{v}^a_\pm$, $\mc{L}_k \wt{v}^a_0 = 0$, $c$ is some constant, and $\mu \ind{^a_b} \wt{D}_a \wt{v}_0^b - \frac{1}{n} \wt{D}_c \wt{v}^c_0 = 0$ with $\mu_{ab} = \wt{D}_{[a} k_{b]}$, with respect to any Patterson--Walker metric. Further, $\wt{v}_0^a$, $\wt{v}_+^a$ and $\wt{v}_-^a$ can be expressed as the lifts \eqref{eq-lift0}, \eqref{eq-lift+} and \eqref{eq-lift-} respectively, where
\begin{enumerate}[(a)]
\item $v^A = \frac{1}{2} \chi^{aA} \wt{D}_a \left( k_b  \wt{v}_{0}^b \right)$ is an infinitesimal projective symmetry, i.e.\ satisfies \eqref{eq-proj-sym}.
\item $w^{AB} = \frac{1}{2} \chi^{aA} \chi^{B}_b \wt{D}_{a} \wt{v}_{+}^b$ satisfies \eqref{eq-bivec} together with the integrability condition \eqref{eq-int-cond-bivec}.
\item $\alpha_A = \check{\eta}_{aA}  \wt{v}_{-}^a$ satisfies the Killing equation \eqref{eq-proj-Killing}.
\end{enumerate}
\end{prop}

\begin{proof}
We work with a Patterson--Walker metric $g_{ab}$ and the relation  \eqref{nup} throughout. Following the argument given in the proof of Proposition \ref{prop-Einstein2lifts}, we first show that for any conformal Killing field $\wt{v}^a$,
\begin{align}\label{eq-to-show}
\mc{L}_k \left( \mc{L}_k -2 \right) \left( \mc{L}_k + 2 \right) \wt{v}^a & = 0 \, .
\end{align}
Differentiating \eqref{eq-conf-iso-prol1} once and substituting \eqref{eq-conf-iso-prol2} and \eqref{eq-conf-iso-prol3} yield
\begin{align}
\wt{D}_a \wt{D}_b \wt{v}_c & = - \mb{g}_{a[b} \wt{\beta}_{c]} - 2 \Rho_{a[b} \wt{v}_{c]} + \wt{v}^d \wt{W}_{dabc} - \mb{g}_{bc} \Rho_{ad} \wt{v}^d - \mb{g}_{bc} \wt{\beta}_a \, . \label{eq-DDs}
\end{align}
Now, using \eqref{eq-Lie-density} with $w=0$ gives
\begin{align*}
\mc{L}_k \wt{v}_a & = k^b \wt{D}_b \wt{v}_a + \wt{v}_b \mu \ind{_a^b} - \wt{v}_a \, .
\end{align*}
We note that $\mc{L}_k \mu \ind{_a^b} = 0$ and using \eqref{eq-DDs}, $k^a k^b \wt{D}_a \wt{D}_b \wt{v}_c = 0$, where we have made use of the fact that $k^a \wt{W}_{abcd} k^d= 0$, and, for a Patterson--Walker metric, $\wt{\Rho}_{ab}k^b =0$ --- see \eqref{eq-Walker-Weyl} and \eqref{eq-Walker-Rho}. Then we compute
\begin{align*}
\mc{L}_k^2 \wt{v}_a & = 2 \left( k^d \wt{D}_d \wt{v}_c \right) \mu \ind{_a^c} - 2 \, \wt{v}_c \mu \ind{_a^c} + 2 \, \wt{v}_a \, , &
\mc{L}_k^3 \wt{v}_a & = 4 \, \mc{L}_k \wt{v}_a \, ,
\end{align*}
which is equivalent to \eqref{eq-to-show}. The result follows immediately.

Next, we write $\wt{v}^a = \wt{v}^A \check{\eta}^a_A + \wt{\alpha}_A \chi^{aA}$. Then contracting \eqref{eq-DDs} with three vertical fields yields
\begin{align*}
 \chi^{aA} \chi^{bB} \chi^{cC} \wt{D}_a \wt{D}_b \wt{v}_c & = \parderv{{}^2}{p_A p_B} \wt{v}^C = 0 \, , & \mbox{i.e.}  & &  \wt{v}^A & = w^{AB} p_B + \psi^A \, ,
\end{align*}
where $w^{AB}$ and $\psi^A$ only depend on $x^A$. Similarly,
\begin{align*}
 \chi^{aA} \chi^{bB} \chi^{cC} \check{\eta}^d_D \wt{D}_a \wt{D}_b \wt{D}_c \wt{v}_d & = \parderv{{}^3}{p_A p_B p_C} \wt{\alpha}_D = 0 \, , \\
  \mbox{i.e.}  \quad  \wt{\alpha}_A & = \psi_A^{BC} p_B p_C + \varphi_A^B p_B +  \alpha_A \, ,
\end{align*}
where $\psi_A^{BC}=\psi_A^{(BC)}$, $\varphi_A^B$ and $\alpha_A$ only depend on $x^A$.

Now, applying Lemma \ref{lem-hom} gives the following conditions:
\begin{enumerate}[(a)]
\item if $\mc{L}_k \wt{v}^a = 0$, then $\wt{v}^A = \psi^A$ and $\wt{\alpha}_A = \varphi_A^B p_B = -\phi_A^B p_B + \frac{n-1}{n+1} \psi p_A $ where $\phi_C^C=0$ and $\psi = \frac{n+1}{n(n-1)} \varphi_C^C $ with factors chosen for later convenience; \label{case-a}
\item if $\mc{L}_k \wt{v}^a = 2 \, \wt{v}^a$, then $\wt{v}^A = w^{AB} p_B$ and $\wt{\alpha}_A = \psi_A^{BC} p_B p_C$;\label{case-b}
\item if $\mc{L}_k \wt{v}^a = -2 \, \wt{v}^a$, then $\wt{v}^A = 0$ and $\wt{\alpha}_A = \alpha_A (x)$;\label{case-c}
\end{enumerate}
In case (\ref{case-a}), we immediately conclude that $\wt{v}^a$ takes the form \eqref{eq-lift0}, while in (\ref{case-c}) that $\wt{v}^a$ takes the form \eqref{eq-lift-}. For case (\ref{case-b}), we return to the conformal Killing equation and equation \eqref{eq-DDs}, and find
\begin{align*}
\chi^{a(A} \chi^{bB)} \wt{D}_a \wt{v}_b & = w^{(AB)} = 0 \, , & \mbox{i.e.} & & w^{AB} & = w^{[AB]} \, , \\
\chi^{a(A} \chi^{bB)} \check{\eta}^c_C \wt{D}_a \wt{D}_b \wt{v}_c & =  \psi_C^{AB} = - 2 \, \delta_C^{(A} \wt{\beta}_a \chi^{aB)}  \, , & \mbox{i.e.} & & \psi_C^{AB} & = \delta_C^{(A} \nu^{B)} \, ,
\end{align*}
for some $\nu^A$, from which it follows that $\wt{v}^a$ takes the form \eqref{eq-lift+}. At this stage, we do not know that $v^A$, $w^{AB}$ and $\alpha_A$ satisfy \eqref{eq-proj-sym}, \eqref{eq-bivec} and \eqref{eq-proj-Killing} respectively, nor that $v^A$, $\phi_A^B$ and $\psi$ are related by \eqref{eq-proj-sym-xtra}, and $w^{AB}$ and $\nu^A$ by \eqref{eq-bivec-xtra}.

Next, we note that by Lemma \ref{lem-lift-properties}, $\wt{v}_+^a$ and  $\wt{v}_-^a$ are tangent to the distributions $U$ and $V$ annihilated of $\eta_A$ and $\chi^{A'}$ respectively. Since $k^a$ is tangent to both then $\wt{v}_\pm^a$ could potentially be of the form $f \, k^a$ for some smooth function $f$. So suppose that $\wt{v}_\pm^a = f \, k^a$. Then $\mc{L}_k \wt{v}_\pm^a = \pm 2 \, \wt{v}_\pm^a$ tells us that $f$ must be non-constant. But since $k^a$ is a conformal Killing field, $f$ must necessarily be constant. Hence, $\wt{v}_\pm^a$ cannot be proportional to $k^a$.

Finally, suppose $\wt{v}_0^a = c \, k^a$ for some constant $c$. Then $\mc{L}_k \wt{v}_0^a = 0$. But computing $\mu \ind{^a_b} \wt{D}_a \wt{v}_0^b - \frac{1}{n} \wt{D}_d \wt{v}^d_0 = - 2 \, c \, (n+1)$ leads to a contradiction. Hence, $\wt{v}_0^a$ cannot be proportional to $k^a$.

To conclude the proof, we show that $v^A$, $\phi_A^B$ and $\psi$ are related by \eqref{eq-proj-sym-xtra}, and $w^{AB}$ and $\nu^A$ by \eqref{eq-bivec-xtra}, and that $v^A$, $w^{AB}$ and $\alpha_A$ satisfy \eqref{eq-proj-sym}, \eqref{eq-bivec} and \eqref{eq-proj-Killing} respectively.
\begin{enumerate}[(a)]
\item Suppose $\mc{L}_k \wt{v}^a = 0$ and $\mu^{ab} \wt{D}_a \wt{v}_b  - \frac{1}{n} \wt{D}_c \wt{v}^c = 0$ so that  $\wt{v}^a=\wt{v}^a_0$ given by \eqref{eq-lift0}. Computing $\wt{D}_a \wt{v}_b$ gives \eqref{eq-projsym2CKf} again. Taking the trace-free part of \eqref{eq-projsym2CKf} yields $\phi_A^B = D_A v^B - \frac{1}{n} \delta_A^B D_C v^C$ and \eqref{eq-proj-sym-prol3}. Now, substituting $\phi_A^B$ into \eqref{eq-proj-sym-prol3} precisely yields \eqref{eq-proj-sym}. Finally,
\begin{align*}
\mu^{ab} \wt{D}_a \wt{v}_b  - \frac{1}{n} \wt{D}_c \wt{v}^c & = \frac{n-1}{n} \left( D_C v^C  - n \psi \right) \, .
\end{align*}
Since, by assumption the left-hand side vanishes, we have $\psi = \frac{1}{n} D_C v^C$.

\item Suppose $\mc{L}_k \wt{v}^a = 2 \, \wt{v}^a$ so that  $\wt{v}^a=\wt{v}^a_+$ given by \eqref{eq-lift+}. Computing $\wt{D}_a \wt{v}_b$ gives \eqref{eq-bivec2CKf} again. The trace-free part of \eqref{eq-bivec2CKf} vanishes, which is equivalent to 
\begin{align}
D_A w^{BC} - \frac{1}{n} D_D w^{BD} \delta_A^C - \delta_A^B \nu^C + \frac{1}{n} \nu^B \delta_A^C & = 0  \, , \label{eq-bivec2CKf1} \\
\left( D_{(A} \nu^{(C} + \Rho_{(A|E} w^{E(C} \right) \delta_{|B)}^{D)} - w^{E(C} W \ind{_{E(A}^{D)}_{B)}} & = 0 \, . \label{eq-bivec2CKf2}
\end{align}
The symmetric part of \eqref{eq-bivec2CKf1} yields $\nu^A =\frac{1}{n-1} D_C w^{CA}$ and the skew-symmetric part reduces to \eqref{eq-bivec}. In particular, $w^{AB}$ satisfies \eqref{eq-bivec}. This in turn implies \eqref{eq-bivec-prol}, which, substituted into \eqref{eq-bivec2CKf2}, yields
\begin{align*}
\left( \frac{1}{2(n-2)} w^{EF} W \ind{_{EF}^{(C}_{(A}} \right) \delta_{B)}^{D)} + w^{E(C} W \ind{_{E(A}^{D)}_{B)}} & = 0 \, .
\end{align*}
Taking the trace shows that both terms vanish separately, as required.

\item Suppose $\mc{L}_k \wt{v}^a = -2 \, \wt{v}^a$ so that $\wt{v}^a=\wt{v}^a_-$ given by \eqref{eq-lift-}. Computing $\wt{D}_a \wt{v}_b$ gives \eqref{eq-projKilling2CKf} again, from which we immediately conclude that $\alpha_A$ is Killing.

\end{enumerate}
\end{proof}

We end the section with a geometric interpretation of a light-like conformal Killing field $\wt{v}^a$ with $\mc{L}_k \wt{v}^a = 0$.
\begin{prop}\label{prop-geod-proj-sym}
Let $\wt{v}^a$ be a conformal Killing field on $\wt{M}$ such that $\mc{L}_k \wt{v}^a = 0$ with associated infinitesimal projective symmetry $v^A$ as in Proposition \ref{prop-conf-iso2lift}. Then $\wt{v}^a$ is light-like if and only if
\begin{align}\label{eq-v-light-geod}
v^B D_B v^A  & = \frac{2}{n+1} ( D_C v^C ) v^A \, ,
\end{align}
for any affine connection $D_A$ in the projective class on $M$. In particular, if $\wt{v}^a$ is light-like then $v^A$ is geodetic.
\end{prop}

\begin{proof}
We compute the norm of the lift $\wt{v}^a=\wt{v}^a_0$ as defined in Lemma \ref{lem-lifts}:
\begin{align*}
\wt{v}^a \wt{v}_a & = 2 \left( \frac{n-1}{n+1} \psi \, v^A - \phi_B^A v^B \right) p_A \, ,
\end{align*}
where we have used  \eqref{nup} as before. So, $\wt{v}^a \wt{v}_a = 0$ if and only if $\phi_B^A v^B= \frac{n-1}{n+1} \psi v^A$. Since $v^A$ is an infinitesimal projective symmetry, we know that $\phi_A^B$ and $\psi$ are given by \eqref{eq-proj-sym-xtra} so that $\wt{v}^a$ is light-like if and only if $v^A$ satisfies \eqref{eq-v-light-geod}.
This condition in particular implies that $v^A$ is geodetic with respect to $D_A$, and thus with respect to the projective structure.
\end{proof}

\subsection{Decomposition of Killing fields of Patterson--Walker metrics}
We now consider the Patterson--Walker metric $g$ induced by a given affine connection $D$ on $M$. Let $\wt{v}^a$ be an infinitesimal symmetry of $g$, i.e.\ $\mc{L}_{\wt{v}}g=0$, which is well-known to be equivalent to the overdetermined equation
\begin{align}
  \wt{D}_{(a}\wt{v}_{b)} & = 0 \, . \label{eq-Killing}
\end{align}
Such a field is also known as a Killing field. We want to understand how $\wt{v}$ decomposes in terms of objects on the affine structure $(M,D)$ in analogy to Proposition \ref{prop-conf-iso2lift} and Theorem \ref{thm-confkillingfields}.

Before we proceed, we recall the definition of an \emph{infinitesimal affine symmetry} as a vector field $v^A$ that preserves the affine structure, i.e.\ it satisfies \eqref{eq-proj-symG} with $\Upsilon_A =0$. Following \cite{Yorozu1983,eastwood-notes}, one can check that such a vector field satisfies the overdetermined second order equation 
\begin{align}\label{eq-affine-sym}
D_{A} D_{B} v^C + v^D R \ind{_{DA}^C_{B}} & = 0 \, .
\end{align}
One can show that \eqref{eq-affine-sym} is equivalent to the system
\begin{align}\label{eq-prol-aff-sym}
D_A v^B - \phi_A^B - \delta_A^B \psi  & = 0 \, , &
D_{A} \phi_{B}^C + v^D R \ind{_{DA}^C_{B}} & = 0 \, , &
D_A \psi  & = 0 \, ,
\end{align}
where we have set $\phi_A^B := D_A v^{B} - \frac{1}{n} D_C v^C \delta_A^B$ and $\psi := \frac{1}{n} D_C v^C$.

Let us define the following vector fields on $\wt{M}$:
\begin{align}
 \wt{v}_0^a & :=v^A \check{\eta}_A^a  - \sqrt{2} \, \phi_B^A \eta_A \chi^{aB} - \frac{1}{2} \psi k^a \,  , \label{eq-lift0aff}
\\
   \wt{v}_+^a & := \sqrt{2} \, w^{AB} \eta_A \check{\eta}^a_B \, , \label{eq-lift+aff}
 \\
       \wt{v}_-^a & := \alpha_A \chi^{aA} \, , \label{eq-lift-aff} 
\end{align}
where $v^A$, $\phi_A^B$, $\psi$, $w^{AB}$ and $\alpha_A$ are tensor fields on $M$, with $w^{AB}=w^{[AB]}$ and $\phi_C^C=0$. One then easily checks that an infinitesimal affine symmetry $v^A$, a parallel bivector $w^{AB}$ and a Killing $1$-form $\alpha_A$ give rise to Killing fields via the lifts \eqref{eq-lift0aff}, \eqref{eq-lift+aff} and \eqref{eq-lift-aff} respectively.

\begin{rema}
Had we lifted an infinitesimal affine symmetry $v^A$ by means of \eqref{eq-lift0}, we would have discovered that $\wt{v}^a_0$ is a homothety with $\wt{D}_a \wt{v}^a_0 = \frac{2n^2}{n+1} \psi$. Since $k^a$ is a homothety, we can modify \eqref{eq-lift0} by adding the term $-\frac{n}{n+1} k^a$ to it and thus obtain the Killing field \eqref{eq-lift0aff}.
\end{rema}

\begin{prop}\label{prop-iso2lift}
A Killing field $\wt{v}^a \in \wt{\mc{E}}^a$  can be uniquely decomposed as
\begin{align}\label{eq-iso2lift}
\wt{v}^a & = \wt{v}^a_0 + \wt{v}^a_- + \wt{v}^a_+ \, ,
\end{align}
where $\mc{L}_k \wt{v}_\pm^a = \pm 2 \, \wt{v}^a_\pm$, $\mc{L}_k \wt{v}^a_0 = 0$.
Further, $\wt{v}_0^a$, $\wt{v}_+^a$ and $\wt{v}_-^a$ can be expressed as
the lifts \eqref{eq-lift0aff}, \eqref{eq-lift+aff} and \eqref{eq-lift-aff}
respectively, where
\begin{enumerate}[(a)]
\item $v^A = \frac{1}{2} \chi^{aA} \wt{D}_a \left( k_b  \wt{v}_{0}^b \right)$ is an infinitesimal affine symmetry, i.e.\ satisfies \eqref{eq-affine-sym}.
\item $w^{AB} = \frac{1}{2} \chi^{aA} \chi^{B}_b \wt{D}_{a} \wt{v}_{+}^b$ is parallel, i.e.\ $D_C w^{AB} =0$, and satisfies the integrability condition $w^{B(A} R\ind{_{B(C}^{D)}_{E)}}=0$.
\item $\alpha_A = \check{\eta}_{aA}  \wt{v}_{-}^a$ satisfies the Killing equation \eqref{eq-proj-Killing}.
\end{enumerate}
\end{prop}

\begin{proof}
Since every Killing field of $g$ is in particular a conformal Killing field with respect to the conformal Patterson--Walker metric $[g]=\mb{c}$, we can recycle the proof of Proposition \ref{prop-conf-iso2lift}. In particular, we obtain the decomposition \eqref{eq-iso2lift}.  Note that unlike in decomposition \eqref{eq-conf-iso2lift}, the homothety $k^a$ does not occur in \eqref{eq-iso2lift} since $k^a$ is not a Killing field. Next, following the same reasoning, we deduce that $\wt{v}^a_0$, $\wt{v}^a_+$ and $\wt{v}^a_-$ take the forms \eqref{eq-lift0aff}, \eqref{eq-lift+aff} and \eqref{eq-lift-aff}. The only difference here is the choice of factors in \eqref{eq-lift0aff}. Finally, we compute $\wt{D}_{(a} \wt{v}_{b)}=0$. When $\wt{v}^a = \wt{v}^a_0$, we find
\begin{multline}\label{eq-affsym2Kf}
\wt{D}_{(a} \wt{v}_{b)} = 
\left( D_A v^{B} - \frac{1}{n} D_C v^C \delta_A^B  - \phi_A^B \right) \chi_{(a}^A \check{\eta}_{b)B} 
+ \frac{1}{2} \left( \frac{1}{n} D_C v^C - \psi \right) g_{ab} \\ 
- \Bigl( D_A \phi_B^C + v^D R \ind{_{DA}^C_B} 
 +  \delta_B^C D_A \psi \Bigr) p_C \chi_{(a}^A \chi_{b)}^B  \, ,
\end{multline}
which tells us that $v^A$ is an infinitesimal affine symmetry, as can be checked directly from the defining equations \eqref{eq-Killing} and \eqref{eq-affine-sym}. When $\wt{v}^a = \wt{v}^a_+$, \eqref{eq-bivec2CKf} with $\nu^A=0$ implies that $w^{AB}$ is parallel. When $\wt{v}^a = \wt{v}^a_-$, \eqref{eq-projKilling2CKf} gives us that $\alpha_A$ is Killing.
\end{proof}

 Taken together, we thus obtain Theorem \ref{thm-killingfields}.

\begin{rema} The fact that $k^a$ does not occur in \eqref{eq-iso2lift} allows us to dispense with the additional requirement $\mu \ind{^a_b} \wt{D}_a \wt{v}_0^b - \frac{1}{n} \wt{D}_c \wt{v}^c_0 = 0$ given in Proposition \ref{prop-conf-iso2lift}. In fact, if $\wt{v}_0^a$ is given by \eqref{eq-lift0aff}, then $\mu \ind{^a_b} \wt{D}_a \wt{v}_0^b - \frac{1}{n} \wt{D}_c \wt{v}^c_0 = 2n$.
\end{rema}

\begin{rema}
For a vector field $v^A\in \mc{E}^A$, one may consider its Hamiltonian lift to $T^*M$, which is just the vector field corresponding to the 1-form $\d \left(v^Ap_A \right)$ via the symplectic structure $\mu$, see \eqref{eq-symp}, i.e.,
\begin{align*}
v^A \parderv{}{x^A} - p_B \parderv{v^B}{x^A} \parderv{}{p_A} \, .
\end{align*}
The authors of  \cite{dunajski-mettler-projective} showed that if $v^A$ is an infinitesimal affine symmetry of $(M,\nabla)$ then its Hamiltonian lift is a Killing field of $(\wt{M},g)$.
As expected from Theorem \ref{thm-killingfields}, this lift corresponds to the lift $\wt{v}_0^a$ given by \eqref{eq-lift0aff}. This is confirmed by re-expressing $\wt{v}_0^a$ in coordinates using \eqref{eq-for-Josef}, \eqref{eq-Special-K},\eqref{nup} and \eqref{eq-prol-aff-sym}.
\end{rema}

Finally, we give the analogue of Proposition \ref{prop-geod-proj-sym}.
\begin{prop}
Let $\wt{v}^a$ be a Killing field on $\wt{M}$ such that $\mc{L}_k \wt{v}^a = 0$ with associated infinitesimal affine symmetry $v^A$ as in Proposition \ref{prop-iso2lift}. Then $\wt{v}^a$ is light-like if and only if $v^B D_B v^A = 0$, i.e.\ $v^A$ is tangent to affinely parametrised geodesics on $M$.
\end{prop}

\begin{proof}
The proof is completely analogous to that of Proposition \ref{prop-geod-proj-sym}: for a Killing field $\wt{v}^a$ given by \eqref{eq-lift0aff}, we find $\wt{v}^a \wt{v}_a = - 2 \left( \psi \, v^A + \phi_B^A v^B \right) p_A$. The result follows from the definitions of $\phi_A^B$ and $\psi$, see \eqref{eq-prol-aff-sym}.
\end{proof}

\section{Special cases and further remarks}
\subsection{Case $n=2$}
In the special case $n=2$, the projective volume form $\bm{\upvarepsilon}_{AB} \in \mc{E}_{[AB]}(3)$ on $(M, \mb{p})$, with inverse $\bm{\upvarepsilon}^{AB} \in \mc{E}^{[AB]}(-3)$, allows us to identify $\mc{E}^A(-1)$ with $\mc{E}_A(2)$, and $\mc{E}^{[AB]}(-2)$ with $\mc{E}(1)$. In particular, it is straightforward to check that $\xi^A \in \mc{E}^A(-1)$ is a solution of the Euler-type equation \eqref{eq-s+2s} if and only if $\alpha_A :=  \xi^{B} \bm{\upvarepsilon}_{BA} \in \mc{E}_A (2)$ satisfies the Killing equation \eqref{eq-proj-Killing}. Similarly, $w^{AB} \in \mc{E}^{[AB]}(-2)$ is a solution of \eqref{eq-bivec} if and only if $\sigma :=  \frac{1}{2} w^{AB} \bm{\upvarepsilon}_{AB} \in \mc{E} (1)$ is a Ricci-flat scale, i.e.\ if it satisfies \eqref{eq-ars}.

This is also reflected at the level of $(\wt{M}, \mb{c})$: any conformal Killing vector field $\tilde{v}^a_\pm$ with $\mc{L}_k \tilde{v}^a_\pm = \pm 2 \, \tilde{v}^a_\pm$ gives rise to an almost Einstein scale $\tilde{\sigma}_\mp$ with $\mc{L}_k \tilde{\sigma}_\mp = \mp \tilde{\sigma}_\mp$. Conversely, any such Einstein scale arises in this way.

\begin{rema}
  Let us assume that $M$ is a two-dimensional surface equipped with Riemannian metric $g_{AB}$ and Levi-Civita covariant derivative $D_A$, and endowed with a Killing field $\al^A$. Then $D_{A} \al_{B}= \la \, \bm{\upvarepsilon}_{AB}$ for some $\la\in\cinf(M)$. Then $\xi^A :=(\ast \al)^A = \al_B \bm{\upvarepsilon}^{AB}$ satisfies $D_A \xi^B =  \la \, \de_A^B$ and therefore constitutes a (non-trivial) Euler-type field on the projective surface $M$ with projective class $\mb{p}$ spanned by $D$. Clearly, $\xi^A$ and $\alpha^A$ are orthogonal to each other. This remark applies in particular to any  surface of revolution in $\rr^3$  in which case $\alpha^A$ represents the infinitesimal generator of the rotation.
\end{rema}

\subsection{Case $n=3$}
In the special case $n=3$, the projective volume form $\bm{\upvarepsilon}_{ABC} \in \mc{E}_{[ABC]}(4)$ on $(M, \mb{p})$, with inverse $\bm{\upvarepsilon}^{ABC} \in \mc{E}^{[ABC]}(-4)$, allows us to identify $\mc{E}^{AB}(-2)$ with $\mc{E}_A(2)$. One can then easily check that $w^{AB} \in \mc{E}^{[AB]}(-2)$ is a solution of \eqref{eq-bivec} satisfying the integrability condition \eqref{eq-int-cond-bivec}
 if and only if $\alpha_A :=  \frac{1}{2} w^{BC} \bm{\upvarepsilon}_{BCA} \in \mc{E}_A (2)$ satisfies the Killing equation \eqref{eq-proj-Killing},  together with the integrability condition  $\alpha_F \bm{\upvarepsilon}^{FB(A} W \ind{_{B(C}^{D)}_{E)}}=0$.

Correspondingly, any conformal Killing vector $\wt{v}^a_+$ with $\mc{L}_k \wt{v}^a_+ = 2 \,\wt{v}^a_+$ gives rise to a conformal Killing vector $\wt{v}^a_-$ with $\mc{L}_k \wt{v}^a_- =-2 \, \wt{v}^a_-$. The explicit form of this relation is as follows. 
Assume $\wt{v}^a_+$ is a conformal Killing field. Since $D_A \bm{\upvarepsilon}_{BCD}=0$, for any affine connection $D \in \mb{p}$, then it is clear that the pullback $\wt{\varepsilon}_{abc} := \chi_a^A \chi_b^B \chi_c^C \bm{\upvarepsilon}_{ABC}$ satisfies $\wt{D}_a \wt{\varepsilon}_{bcd} =0$ with respect to any Patterson--Walker metric. A short computation then shows that $\tilde{v}_-^a := \frac{1}{2} \wt{\varepsilon}^a{}_{bc} \wt{D}^b \tilde{v}_+^c$ is indeed a conformal Killing field.

\subsection{Contact projective structures in odd dimensions}
There is a specific class of (odd-dimensional) projective structures on $M$ allowing a compatible contact structure. 
According to \cite{FoxIndiana}, these are the projective structures subordinate to the so-called contact projective structures.  
It follows that under a curvature condition imposed on the contact projective structure
(known as the vanishing of the contact torsion) 
one obtains a projective structure $\mb{p}$  on $M$ admitting a Killing 1-form $\al_A$.  
In particular, every projective structure $(M,\mb{p})$ determined by a contact projective structure with vanishing contact torsion gives rise to an infinitesimal conformal symmetry of $(\wt{M},\mb{c})$.

\subsection{Relation to Cartan geometry and tractor calculus}\label{relation}
The original oriented projective structure $(M,\mb{p})$ can be equivalently described as a Cartan geometry of type $(\SL(n+1),P)$ with $P$ a parabolic subgroup of $\SL(n+1)$, and the conformal spin structure $(\wt{M},\mb{c})$ can be equivalently described as a Cartan geometry of type $(\Spin(n+1,n+1),\wt{P})$, with $\wt{P}$ a parabolic subgroup, see \cite{cap-slovak-book}. This viewpoint was used in \cite{hsstz-fefferman} to relate the respective geometries (see also \cite{nurowski-sparling,nurowski-proj} for similar Cartan geometric approaches). The formulation in \cite{hsstz-fefferman} follows the so-called Fefferman-type construction, which is based on a group inclusion $\SL(n+1) \embed \Spin(n+1,n+1)$ of the underlying (Cartan) structure groups. Note that the conformal structure constructed in this way lives on the total space of the weighted cotangent bundle with the zero section removed $T^*M(2) \setminus \{ 0 \}$ rather than on $T^*M(2)$ as in the present article.

The decomposition of conformal Killing fields of $(\wt{M},\mb{c})$ can also be understood in this framework: Conformal Killing fields of $(\wt{M},\mb{c})$ are equivalent to infinitesimal symmetries of the equivalent \emph{Cartan geometry} $(\wt{\G},\wt{\om})$ and according to \cite{cap-infinitaut} those infinitesimal symmetries can be described equivalently by sections of the \emph{conformal adjoint tractor bundle} $\mc{A}\wt{M}$ associated to the adjoint representation of $\Spin(n+1,n+1)$ on $\so(n+1,n+1)$, parallel with respect to a certain connection referred to as the \emph{prolongation connection}. Likewise, projective infinitesimal symmetries are described as suitable parallel sections of the \emph{projective adjoint tractor bundle} $\mc{A}M$. Since $(\wt{M},\mb{c})$ is (locally) induced in a natural way from the projective structure $(M,\mb{p})$, the adjoint tractor bundle decomposes naturally according to the
decomposition of $\so(n+1,n+1)$ into its $\SL(n+1)$-irreducible components
\begin{align*}
  \rr \oplus \sl(n+1) \oplus \La^2\rr^{n+1}\oplus \La^2(\rr^{n+1})^*.
\end{align*}
Decomposing an infinitesimal symmetry into its constituents with respect to this decomposition and reinterpreting the resulting sections on the original projective structure $(M,\mb{p})$ gives an alternative (algebraic) approach to derive Theorem \ref{thm-confkillingfields}.

Let us illustrate this formalism within the general approach of the present article. A choice of metric $g$ in $\mb{c}$ splits the adjoint tractor bundle as $\mc{A}\wt{M} \cong \wt{\mc{E}}_a [2] \oplus \wt{\mc{E}}_{ab} [2] \oplus \wt{\mc{E}} \oplus \wt{\mc{E}}_a$. Similarly, a choice of torsion-free affine connection $D$ in $\mb{p}$ splits the projective adjoint tractor bundle, which is associated to $\sl(n+1)$ as $\mc{A} M \cong \mc{E}^A \oplus \left( \mc{E}^A_B \oplus \mc{E} \right) \oplus \mc{E}_B$. 
A conformal Killing field $\wt{v}^a$ can then be expressed as a section $\wt{\Sigma} = (\wt{v}_a, \wt{\phi}_{ab}, \wt{\psi} , \wt{\beta}_a)$, where $\wt{\phi}_{ab}$, $\wt{\psi}$ and $\wt{\beta}_a$ were defined at the beginning of section \ref{sec-infsym}. The defining equation \eqref{eq-conf-iso-prol1} together with its prolongation \eqref{eq-conf-iso-prol3}, \eqref{eq-conf-iso-prol2} and \eqref{eq-conf-iso-prol4} then can be understood equivalently as $\wt{\Sigma}$ being parallel with respect to the prolongation connection on $\mc{A}\wt{M}$. Similarly,
an infinitesimal projective symmetry $v^A$ can be expressed as a section $\Sigma = (v^A, \phi_B^A, \psi, \beta_A)$, where $ \phi_B^A$, $\psi$ and $\beta_A$ were defined at the beginning of section \ref{sec-pro-inv-eq}. The defining equation \eqref{eq-proj-sym} together with its prolongation \eqref{eq-proj-sym-prol3} can be interpreted as $\Sigma$ being parallel with respect to the prolongation connection on $\mc{A} M$. The relation between $\wt{\Sigma}$ and $\Sigma$ is given in terms of the lift $\tilde{v}^a_0$ of Lemma \ref{lem-lifts}.
An analogous approach can be employed to describe almost Einstein scales on $(\wt{M},\mb{c})$ in terms of parallel sections of the \emph{standard tractor bundle} and relate them to projective data.

This Cartan geometric approach can be employed to relate a wider class of invariant overdetermined equations on the respective projective and conformal structures.

\def\polhk#1{\setbox0=\hbox{#1}{\ooalign{\hidewidth
  \lower1.5ex\hbox{`}\hidewidth\crcr\unhbox0}}}

\end{document}